\documentclass[11pt]{amsart}

\usepackage{amsmath}
\usepackage{amssymb}
\usepackage{amsthm}
\usepackage{amsfonts}
\usepackage{mathrsfs}
\usepackage{hyperref}
\usepackage{xcolor}
\hypersetup{colorlinks=true,
linktoc=all,linkcolor=black,
citecolor=black}
\usepackage{mathtools}
\usepackage[
textwidth=410pt,
textheight=620pt,
hmarginratio=1:1,
vmarginratio=1:1]{geometry}

\theoremstyle{plain}
\newtheorem{thm}{Theorem}[section]
\newtheorem{cor}[thm]{Corollary}
\newtheorem{lem}[thm]{Lemma}

\newtheorem{prop}[thm]{Proposition}

\theoremstyle{definition}

\newtheorem{problem}[thm]{Problem}

\theoremstyle{remark}
\newtheorem{rem}[thm]{Remark}

\newcommand{\eqd}{:=}
\newcommand{\s}{\vartheta}
\newcommand{\n}{m}
\newcommand{\Erf}{\operatorname{erf}}

\newcommand{\hatR}{V}

\newcommand{\ErfN}{\Erf(2\sqrt{\n}\hatR)}

\newcommand{\ExpN}{\operatorname{e}^{-2\n \hatR^2}}

\newcommand{\Pot}{\mathscr{U}_1}
\newcommand{\PotV}{\mathscr{U}_0}
\newcommand{\PotB}{\mathscr{B}}
\newcommand{\PotA}{\mathscr{A}}
\newcommand{\Potu}{\mathscr{U}}

\newcommand{\vDelta}{\varDelta} 
\newcommand{\dbar}{\bar\partial}
\newcommand{\coeffH}{\widehat{H}}
\newcommand{\coeffG}{\widehat{G}}
\newcommand{\coeffh}{\widehat{h}}
\newcommand{\coeffE}{\widehat{E}}
\newcommand{\coeffcalE}{\widehat{\mathcal{E}}}
\newcommand{\coeffL}{\widehat{L}}
\newcommand{\coeffSop}{\widehat{\Sop}}
\newcommand{\E}{\mathcal{E}}
\newcommand{\C}{\mathbb{C}}
\newcommand{\D}{\mathbb{D}}
\newcommand{\T}{\mathbb{T}}
\newcommand{\N}{\mathbb{N}}
\newcommand{\R}{\mathbb{R}}
\newcommand{\A}{\mathbb{A}}

\newcommand{\calS}{\mathcal{S}}

\newcommand{\calE}{\mathcal{E}}
\newcommand{\diffA}{\mathrm{dA}}
\newcommand{\diffs}{\mathrm{ds}}
\newcommand{\diff}{\mathrm{d}}
\newcommand{\Pop}{\mathbf{P}}
\newcommand{\Tope}{\mathbf{T}}
\newcommand{\Sop}{\mathbf{S}}

\newcommand{\Iop}{\mathbf{I}}
\newcommand{\Rop}{\mathbf{R}}
\newcommand{\say}[1]{``#1''}
\newcommand{\e}{\mathrm{e}}
\newcommand{\Ordo}{\mathrm{O}}
\newcommand{\ordo}{\mathrm{o}}
\newcommand{\imag}{\mathrm{i}}
\renewcommand{\Re}{\mathrm{Re}\,}
\renewcommand{\hat}{\widehat}

\newcommand{\circledQ}{Q^\circledast}

\numberwithin{equation}{section}

\begin{document}

\title[Berezin density and planar orthogonal polynomials]
{Berezin density and planar orthogonal polynomials}

\author{Haakan Hedenmalm}

\address{Department of Mathematics, KTH Royal Institute of Technology, 
Stockholm, S -- 100 44 Stockholm, Sweden}
\email{haakanh@math.kth.se}

\author{Aron Wennman}
\address{Department of Mathematics, KTH Royal Institute of Technology, 
Stockholm, S -- 100 44 Stockholm, Sweden}
\email{aronw@math.kth.se}

\keywords{Bergman kernel,
Riemann-Hilbert problem, planar orthogonal polynomials}

\subjclass{Primary 42C05, 35C20, 41A60, secondary 35Q15, 35B40, 60B20}

\begin{abstract}
We introduce a nonlinear potential theory problem for the Laplacian, 
the solution of which characterizes 
the Berezin density $B(z,\cdot)$ for the polynomial Bergman space,
where the point $z\in\C$ is fixed.
When $z=\infty$, the Berezin density is
expressed in terms of the squared modulus of the corresponding 
normalized orthogonal polynomial $P$.
We use an approximate version of this characterization to 
study the asymptotics of the orthogonal polynomials
in the context of exponentially varying weights.
This builds on earlier works by Its-Takhtajan and by the first author on a 
soft Riemann-Hilbert problem for planar orthogonal polynomials,
where in place of the Laplacian we have the $\bar\partial$-operator.
We adapt the soft Riemann-Hilbert approach to 
the nonlinear potential problem, where the nonlinearity is 
due to the appearance of $|P|^2$ in place of $\overline{P}$.  
Moreover, we suggest how to adapt the potential theory method to the study of
the asymptotics of more general Berezin densities $B(z,w)$
in the off-spectral regime, that is, when $z$ is fixed
outside the droplet.
This is a first installment in a program to obtain an explicit global expansion 
formula for the polynomial Bergman kernel, and, in particular, of the one-point function
of the associated random normal matrix ensemble.
\end{abstract}
\maketitle

\section{Overview}
We consider a novel potential theoretic problem involving the Laplacian
and the Berezin density associated with a polynomial Bergman kernel
for exponentially varying weights in the plane. The Berezin density 
for the point at infinity has a particularly simple form, which allows us
to use this framework to characterize the planar orthogonal polynomials.
In particular, we supply an independent proof of the main results of
\cite{HW-ONP, HedenmalmPotential} aimed at obtaining
an asymptotic expansion of the associated 
polynomial Bergman density.
The method is reminiscent of the Riemann-Hilbert approach to orthogonal
polynomials on the real line, and is closely related to the
$\bar\partial$-approach introduced by Its-Takhtajan, and developed
further by the first-named author in \cite{HedenmalmPotential},
where it was called a \emph{soft Riemann-Hilbert problem}.
We present the potential problem for the point at infinity in \S\ref{s:potential-OP}.
In \S\ref{s:Berezin} we indicate how to modify the 
approach to analyze the Berezin density for more 
general (off-spectral) points. At the same time, we also mention an interpretation of
the Ward identity in terms of the potential for the Berezin density, and
look at the zero set of the polynomial Bergman kernel and its motion 
under holomorphic variation of the second variable. The proofs that the
potential problems indeed characterize the Berezin densities are carried
out in \S\ref{s:pf-potential-characterization}.
In \S\ref{s:approximate}, we consider an approximate potential problem 
for the asymptotic analysis of the orthogonal polynomials.
Moreover, in \S\ref{s:comput}--\ref{s:sol-master} we supply an
algorithm which obtains a formal solution in the sense of asymptotic expansions
to the potential theoretic problem. We carry this out in detail in the instance
of the Berezin density for the point at infinity, but any off-spectral point
could be analyzed in an analogous fashion.
To make the formal solution rigorous, we estimate the size of the involved quantities
in the asymptotic expansions, and use a standard $\bar\partial$-surgery
technique based on H{\"o}rmander estimates. This is carried out
in \S\ref{s:proof}--\ref{s:truncation}. 

For the convenience of the reader, we point out that the expansion formula
for the orthogonal polynomials is obtained in Theorem~\ref{thm:main-prel} and 
Theorem~\ref{thm:main}, see also Lemma~\ref{lem:truncation-h-prel} for the truncation to
a finite asymptotic expansion.

\section{Orthogonal polynomials through a potential problem}
\label{s:potential-OP}
\subsection{The exact potential problem}
In this section, we explore the potential problem for
planar orthogonal
polynomials, where the orthogonality is taken
with respect to the exponentially varying measure
$\e^{-2mQ}\diffA$
on the plane. We are interested in the asymptotics as 
the degree $n$ and the parameter $m$ grow to infinity
in a proportional fashion: $n=\tau m$. We use the notation $P=P_{m,n}$
for the $L^2$-normalized orthogonal polynomial, and $\pi=\pi_{m,n}$ for its monic
cousin. Here, $\diffA=(2\pi\imag )^{-1}\diff\bar z\wedge\diff z$ is the planar area element, 
and $Q$ is a $C^2$--smooth weight function subject to the growth condition
\begin{equation}
\label{eq:growth-Q}
\liminf_{|z|\to+\infty}\frac{Q(z)}{\log|z|}> 1,
\end{equation}
at infinity. 
We refer to \S\ref{s:assumptions} below for detailed additional assumptions and
for basic notation.

As in the classical Riemann-Hilbert approach to orthogonal
polynomials on the real line, and also in \cite{HedenmalmPotential}, 
the current method is based on the observation that the solvability of 
a certain system of PDEs determine the 
orthogonal polynomial uniquely. This system may be solved
(at least approximately) in a constructive manner,
and the resulting approximate solution 
must be close to the orthogonal polynomial in question, in a precise sense
specified below. 

\medskip

The potential problem, which inspires our approach, 
reads as follows. We use the notation 
$\vDelta=\vDelta_z=\partial_z\bar\partial_z=\frac14(\partial^2_x+\partial^2_y)$ 
for the quarter Laplacian, where
$\partial=\partial_z$ and $\bar\partial=\bar\partial_z$ are the standard Wirtinger 
derivatives, defined below in \eqref{eq:Wirtinger},
and where $z=x+\imag y$.
\begin{problem}
\label{problem:1}
Determine a pair $(\PotV,\pi)$ of functions on $\C$ such that
$\PotV$ is real-valued, and such that 
\begin{equation}\label{eq:pot-exact}
\begin{cases}
\vDelta \PotV(z)=|\pi(z)|^2\e^{-2m Q(z)},\\
\bar\partial \pi\equiv 0,\\
\pi^{-1}\partial\, \PotV\in C^2(\C),
\end{cases}
\end{equation}
with asymptotic behavior
\begin{equation}
\label{eq:growth-Upi}
\big(\PotV(z),\pi(z)\big)=
\big(\kappa^2\log|z|^2+\Ordo(|z|^{-1}),\,z^n(1+\Ordo(|z|^{-1})\big)
\end{equation}
as $|z|\to+\infty$, for some positive constant $\kappa=\kappa_{m,n}$.
\end{problem}

We have the following elementary observation.
\begin{prop}
\label{lem:exact}
If $\n$ is sufficiently large and if $n\le \n$, 
Problem~\ref{problem:1} admits a unique solution $(\PotV,\pi)$.
The function $\pi=\pi_{m,n}$ is then the $n$-th monic orthonormal
polynomial in $L^2(\C,\e^{-2m Q}\diffA)$, and $\PotV$ is the 
associated logarithmic potential
\[
\PotV(z)=\int_{\C}g(z,w)\, |\pi_{m,n}(w)|^2\e^{-2mQ(w)}\diffA(w),
\]
where $g(z,w)=\log|z-w|^2$ is the fundamental solution for the operator
$\vDelta$ with respect to the normalized area measure $\diffA$. 
Moreover, we have that
\[
\kappa_{m,n}^2=\lVert\pi_{m,n}\rVert_{2mQ}^2=
\int_{\C}|\pi_{m,n}(z)|^2\e^{-2\n Q(z)}\diffA(z).
\]
\end{prop}

The proof of Proposition~\ref{lem:exact} is rather elementary, 
but relies on some computations
which we defer to \S\ref{s:exact}.
We try to obtain approximate solutions to Problem~\ref{problem:1}, and 
this allows us derive an asymptotic formula for the planar orthogonal polynomials 
(see Theorems~\ref{thm:main-prel} and \ref{thm:main} below). 
As for the approximate solution,
we will replace the delicate divisibility criterion $\pi^{-1}\partial\,\PotV\in C^2$,
which implies that the gradient of $\PotV$ vanishes on the zero set of $\pi$,
by the requirement that the gradient of $\PotV$ vanishes asymptotically
interiorly of a curve which is characteristic for the problem 
(the boundary of the associated droplet).

Introducing the pair $(\Pot, P)=(\kappa^{-2}\PotV,\kappa^{-1}\pi)$,
we obtain a solution to the $L^2$-normalized problem
\begin{equation}
\label{eq:pot-exact-2}
\begin{cases}
\vDelta\, \Pot(z)=|P(z)|^2\e^{-2m Q(z)},\\
\bar\partial P\equiv 0,\\
P^{-1}\partial\, \Pot\in C^2(\C),
\end{cases}
\end{equation}
with asymptotic behavior
\begin{equation}
\big(\Pot(z),P(z)\big)=
\big(\log|z|^2+\Ordo(|z|^{-1}),\kappa^{-1}z^n(1+\Ordo(|z|^{-1})\big)
\label{eq:prob2b}
\end{equation}
as $|z|\to+\infty$, where $\kappa=\kappa_{m,n}$ is as before. 
It then follows
from Proposition~\ref{lem:exact} that $P=\pi/\lVert \pi\rVert_{L^2(\C,\e^{-2\n Q}\diffA)}$ 
is the  normalized
orthogonal polynomial.

\subsection{Motivation}
\label{s:motiv}
The study of the asymptotic expansions of planar orthogonal
polynomials for exponentially varying weights
is a rather recent development. 
In \cite{HW-ONP}, the asymptotic
expansion of the normalized orthogonal polynomial 
$P_{m,n}$ (and hence $\pi_{m,n}$) was found, and this result was used
to establish boundary universality in the random normal matrix model.
Related results on orthogonal polynomials and Bergman kernels were obtained
in \cite{HW-OFF, HW-ONP2}.
In the work \cite{HedenmalmPotential}, the study of planar orthogonal
polynomials was pursued using a 
\say{soft Riemann-Hilbert}-approach, which involves the equation
\begin{equation}
\label{eq:d-bar}
\bar\partial{\Psi}=\overline{\pi}\,\e^{-2mQ},
\end{equation}
where we recall that $\pi=\pi_{m,n}$ is the $n$-th
monic orthogonal polynomial in the space $L^2(\C,\e^{-2mQ}\diffA)$.

A strong motivation for studying planar orthogonal polynomials
comes from the connection with 
random normal matrix theory, see e.g. \cite{HM}, \cite{WZ}.
This connection comes from the fact that the \emph{polynomial Bergman kernel}
\begin{equation}
\label{eq:kernel-def}
{\rm k}_{m,n}(z,w)=\sum_{k=0}^{n-1}P_{m,k}(z)\overline{P_{m,k}(w)}
\end{equation}
is, up to a weight factor, the correlation kernel for the determinantal
random normal matrix eigenvalue ensemble.
It is natural to analyze the \emph{diagonal restriction}
\[
{\rm k}_{m,n}(z,z)=\sum_{k=0}^{n-1}|P_{m,k}(z)|^2,
\]
in view of the connection with the one-point function 
(or density of states) in the random normal matrix model,
see, e.g.,\ the introduction of \cite{HW-ONP} as well as
\cite{ahm2, ahm3}. Studying the diagonal restriction is also motivated by the fact that
it uniquely determines the full kernel
${\rm k}_{\n,n}(z,w)$ by polarization.
We would like to find a suitable asymptotic expansion for ${\rm k}_{\n,n}(z,z)$
valid throughout the complex plane. This is well understood in the instance when there is
no polynomial truncation, see, e.g.,\ the works \cite{Tian, Catlin, Zelditch, BBS, DHS}
of Tian, Catlin, Zelditch, Berman-Berndtsson-Sj{\"o}strand, and Deleporte-Hitrik-Sj{\"o}strand. 
In addition, the recent progress in understanding each individual 
orthogonal polynomial suggests that such asymptotics may be available also in the truncated case.
To make substantial progress on this issue, we would need to be able to capture each
contribution $|P_{\n,k}|^2$ as well as the effect of summation over the degrees.
This splits the difficulty in two parts: {\rm (a)} to derive asymptotic expansions of each 
contribution $|P_{\n,k}|^2$, and
{\rm (b)} to effectively sum these asymptotic expansions
over $k=0,\ldots, n-1$.

Here, we analyze aspect {\rm (a)} of this program, with an aim to 
use this to make progress on {\rm (b)} later on. In principle, we could just
implement the results from \cite{HW-ONP} and \cite{HedenmalmPotential} which give
the asymptotics for $P_{\n,k}$, take the squared modulus of it for part {\rm (a)} 
and sum the resulting expressions over $k$ for {\rm (b)}. This was carried out
in \cite{HW-ONP} resulting in near-diagonal asymptotics 
which gave rise to a boundary universality result. 
A similar approach was recently carried by Ameur-Cronvall \cite{AmeurCronvall}
in some other regimes, further away from the diagonal. However, given that we need
to sum the resulting expression $|P_{\n,k}|^2$ over $k$, it is much more
convenient to have a calculus which gives the squared moduli $|P_{\n,k}|^2$ directly.
This is the purpose of the present work. 
As for the summation in part {\rm (b)}, given that in many situations integration
is actually easier than summation, one should attempt to find a primitive
with respect to the degree parameter $k$ of $|P_{\n,k}|^2\e^{-2\n Q}$. We believe that this is within reach, 
and that this primitive function combined with the Euler-MacLaurin summation formula 
can be used to approximate the true one-point function.

To characterize the squared modulus $|P_{\n,n}|^2$, 
we appeal to Proposition~\ref{lem:exact} above,
which says that we should study the renormalized version of 
Problem~\ref{problem:1} expressed in \eqref{eq:pot-exact-2} and \eqref{eq:prob2b}. 
This may be viewed as an extension of the $\bar\partial$-approach
to planar orthogonal polynomials, as formulated by 
Its and Takhtajan in \cite{ItsTakhtajan}
and developed in \cite{HedenmalmPotential}. 
If $\mathscr{V}_0 =\partial \mathscr{U}_0$, 
then $\mathscr{V}_0$ solves 
\[
\bar\partial \mathscr{V}_0=|\pi|^2\e^{-2\n Q},
\] 
and, given that $\pi^{-1}\mathscr{V}_0\in C^2$ by assumption, 
we must have $\mathscr{V}_0=\pi\Psi$, 
where $\Psi$ solves the equation \eqref{eq:d-bar}.
Our approach to the study of the 
functions $|P_{\n,n}|^2$ is based on
finding approximate solutions of the renormalized version of 
Problem~\ref{problem:1}, see Theorem~\ref{thm:main}
and Corollary~\ref{cor:main} below 
(cf.\ Proposition~\ref{prop:approx-Pot-eq}).
In particular, lifting to the Laplacian allows 
us to keep the approach real-valued.

\subsection{Related work}
The asymptotic properties of orthogonal polynomials 
on the real line have been successfully analyzed 
using Riemann-Hilbert methods, see e.g.\ the works 
\cite{Its2, Deift-varying, Deift-PNAS} and \cite{DeiftZhou}
by Fokas, Its, Kitaev, Deift, Kriecherbauer, 
McLaughlin, Venakides, and Zhou, in various constellations. 
The $\bar\partial$-approach to planar orthogonal polynomials appeared first by
Its-Takhtajan in \cite{ItsTakhtajan},
and was studied also by Klein--McLaughlin in \cite{KleinMcLaughlin}.

In addition to the above-mentioned recent works \cite{HW-ONP, HW-ONP2, HW-OFF, HedenmalmPotential}
on the asymptotics of planar orthogonal polynomials for rather general potentials,
we mention a related body of work concerning special classes of potentials
which can be studied using classical Riemann-Hilbert methods. 
This includes work by Balogh-Bertola-Lee-McLaughlin \cite{bblm}, 
Bleher-Kuijlaars \cite{BleherKuijlaars}, Grava-Balogh-Merzi \cite{BGM}, 
Bertola-Elias Rebelo-Grava \cite{BEG}, Lee-Yang \cite{LeeYang1, LeeYang2} 
and by Byun-Yang \cite{ByunYang}.

\subsection{Notation}
\label{s:assumptions}
We denote by $\diffA$ and $\diffs$ the usual area and length elements
in the complex plane $\C$, normalized such that
the unit disk $\D$ and the unit circle $\T$ have unit area and length, respectively.
We recall that we use the \say{quarter Laplacian} 
\[
\vDelta\eqd \partial\bar\partial=\tfrac14\big(\partial_x^2+\partial_y^2),
\]
where $\partial=\partial_z$ and 
$\bar\partial=\bar\partial_z$ are the standard Wirtinger
differential operators
\begin{equation}\label{eq:Wirtinger}
\partial=\tfrac12\big(\partial_x -\imag \partial_y\big),
\qquad
\bar\partial=\tfrac12\big(\partial_x +\imag \partial_y\big),
\end{equation}
where $z=x+\imag y$.

Denote by $\mathrm{erf}$ the standard error function, given by
\[
\mathrm{erf}(x)=\frac{1}{\sqrt{2\pi}}
\int_{-\infty}^x \e^{-t^2/2}\diff t.
\]
With this normalization, we have the identity $\mathrm{erf}(x)+\mathrm{erf}(-x)=1$, 
the derivative is given by $\mathrm{erf}'(x)=(2\pi)^{-\frac12}\exp(-t^2/2)$
and we have $\lim_{x\to+\infty}\mathrm{erf}(x)=1$.
We note the bound
\begin{equation}
\label{eq:erf-est1}
0<\mathrm{erf}(x)< (2\pi)^{-\frac12}|x|^{-1}\e^{-x^2/2},
\end{equation}
valid for $x<0$, while for $x>0$ we have
\begin{equation}
\label{eq:erf-est2}
1-(2\pi)^{-\frac12}x^{-1}\e^{-x^2/2}<\mathrm{erf}(x)<1.
\end{equation}

We will denote by $Q$ an external \say{potential}, by which we mean
a $C^2$-smooth function $Q:\C\to \R$, subject to the growth
condition \eqref{eq:growth-Q}.

We denote by $\lVert \cdot\rVert_{2\n Q}$ and 
$\langle \cdot\,,\cdot\rangle_{2\n Q}$ the usual norm and sesquilinear
inner product
in the space $L^2(\C,\e^{-2\n Q}\diffA)$, respectively.

\subsection{Potential theoretic background}
\label{s:pot-theory}
For a parameter $\tau\in (0,1]$, we define the envelope function
\begin{equation}\label{eq:check-Q}
\check{Q}_\tau(z)\eqd\sup\big\{u(z):u\in \mathrm{SH}_\tau(\C),\, 
u(z)\le Q(z)\text{ for all }z\in\C\big\},
\end{equation}
where $\mathrm{SH}_\tau(\C)$ denotes the class of subharmonic
functions which grow at most as $\tau\log|z| + \Ordo(1)$ at infinity.
The function $\check{Q}_\tau$ is $C^{1,1}$-smooth and 
subharmonic; see e.g.\ \cite{HM}.

We let $\calS_\tau$ be the \emph{coincidence set} 
\[
\calS_\tau\eqd\big\{z\in\C:\,\check{Q}_\tau(z)=Q(z)\big\}.
\]
By a Perron-type argument, 
one sees that the function $\check{Q}_\tau$
is harmonic outside $\calS_\tau$.
Since $\tau\le 1$, the set $\calS_\tau$ is compact;
see \cite{SaffTotik}. 
We will moreover assume that $\calS_\tau$ is a connected set, and that
that the outer boundary, i.e.\ the boundary $\Gamma_\tau$ 
of the unbounded component $\Omega_\tau$ of $\C\setminus\calS_\tau$, 
is a real-analytically smooth curve.
In addition, we assume that $Q$ is strictly subharmonic 
(i.e.\ that $\vDelta Q>0$ holds) and real-analytic
in a neighborhood of $\Gamma_\tau$, which in particular implies
that other components of the complement
$\partial\calS_\tau\setminus\Gamma_\tau$ are located at a positive
distance from the curve $\Gamma_\tau$.
We denote by $\phi_\tau$ the standard Riemann mapping of $\Omega_\tau$
onto the exterior disk $\D_\e\eqd\{z\in\C:|z|>1\}$, normalized 
so that $\phi_\tau(\infty)=\infty$ and $\phi_\tau'(\infty)>0$.

The function $\check{Q}_\tau$ is harmonic on the complement 
$\C\setminus\calS_\tau$, so in particular on the unbounded
component $\Omega_\tau$.
Under the present assumptions,
the restriction $\check{Q}_\tau\vert_{\Omega_\tau}$
extends to a harmonic function $\breve{Q}_\tau$ on a larger domain
$\Omega_{\tau,\epsilon}=\{z\in\C: {\rm dist}_{\C}(z,\Omega_\tau)\le \epsilon\}$
for some small positive $\epsilon$. 
Due to well-known properties of the obstacle problem, the function
$Q-\breve{Q}_\tau$ is non-negative on a small neighborhood of $\Omega_\tau$,
which we can take to equal $\Omega_{\tau,\epsilon}$ without loss of generality.
In fact, we have that
\[
(Q-\breve{Q}_\tau)(z)\asymp {\rm dist}_\C(z,\Gamma_\tau)^2
\]
holds in a neighborhood of $\Gamma_\tau$. 
There exists a uniquely determined
$C^2$-smooth square root $V_{\tau}$ of $(Q-\breve{Q}_\tau)(z)$
which is positive outside $\Gamma_\tau$ and negative inside.
The square root is automatically real-analytically smooth 
with non-vanishing gradient in a small neighborhood of $\Gamma_\tau$,
and $C^2$--smooth and strictly positive on $\Omega_\tau$.

\section{The Berezin density and its potential}
\begin{figure}[t!]
\centering
\includegraphics[width=\linewidth]{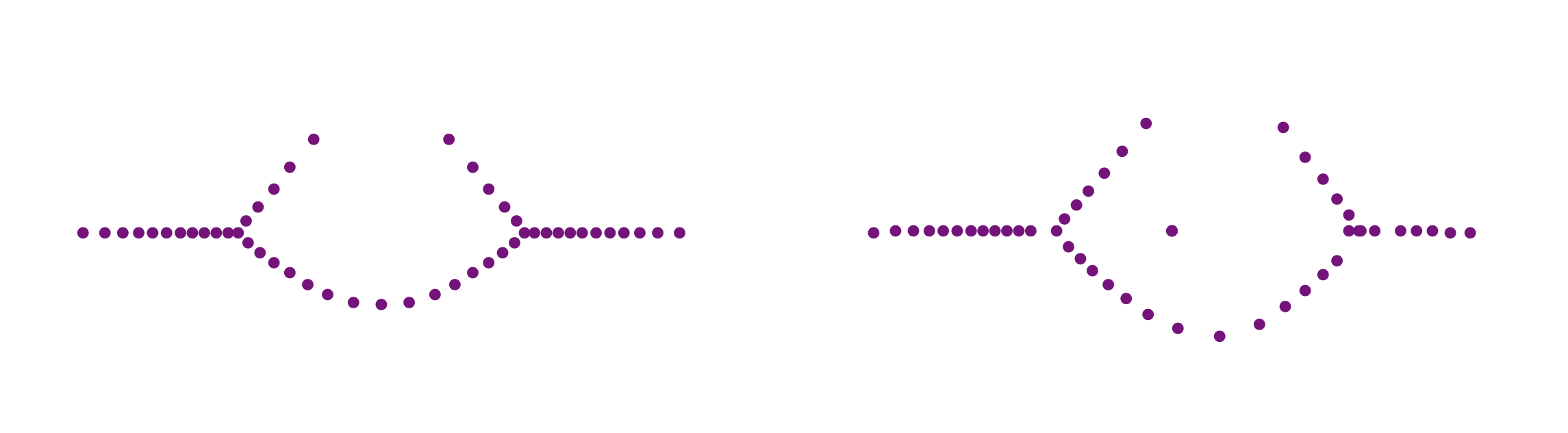}
\caption{Zeros of ${\rm k}_{\n,n}(z,w)$ for fixed $z$, 
for an elliptic potential $Q(z)=\frac12|z|^2+t\Re (z^2)$ with $t=0.6$ and
$(\n,n)=(150,22)$. The point $z$ equals {\rm (a)} $.027\imag$ 
and {\rm (b)} $0.36+0.15\imag$, respectively.}
\label{fig:Berezin-zeros}
\end{figure}

\label{s:Berezin}
\subsection{Polynomial Bergman kernels and their zero sets}
 The polynomial Bergman kernel
$\mathrm{k}(z,w)={\rm k}_{\n,n}(z,w)$ given by \eqref{eq:kernel-def}
is a polynomial of degree $\le n-1$ in $z$ 
(and in $\bar w$ as well, given the symmetry $\mathrm{k}(z,w)=
\overline{\mathrm{k}(w,z)}$). 
As such, the zero locus
\[
\mathfrak{Z}=\mathfrak{Z}_{\n,n}:=\{(z,w)\in\C^2:\,\mathrm{k}(z,w)=0\}
\]
is an algebraic variety in the variables $(z,\bar w)$ 
(and, alternatively, in $(\bar z,w)$). The slices 
\[
\mathfrak{Z}(z)=\mathfrak{Z}_{\n,n}(z):=\{w\in\C:\,\mathrm{k}(z,w)=0\}
\]
generically consist of $n-1$ points if we count multiplicities, and they move about 
holomorphically in $\bar z$. When $z$ approaches infinity, the zero set
$\mathfrak{Z}(z)$ converges to
\[
\mathfrak{Z}(\infty)=\mathfrak{Z}_{\n,n}(\infty):=\{w\in\C:\,P_{n-1}(w)=0\},
\]
where we write $P_{n-1}=P_{m,n-1}$ to simplify the notation. The reason for
this is that we have the convergence
\[
\lim_{|z|\to+\infty}\frac{\mathrm{k}(z,w)}{P_{n-1}(z)}=\overline{P_{n-1}(w)}.
\]
It is an interesting task to study the motion of $\mathfrak{Z}(z)$ as the point
$z\in\C$ varies.  The work here as well as in the work \cite{HW-OFF} suggests 
that for off-spectral points $z\in\C\setminus\calS$ where $\calS=\calS_\tau$ with 
$\tau=n/m$, the zero set
$\mathfrak{Z}(z)$ will typically be confined to the droplet, i.e., we should have
$\mathfrak{Z}(z)\subset\calS$, and then $\mathfrak{Z}(z)$ consists of precisely 
$n-1$ points. The only way to have fewer points than $n-1$ is to have at least
one point escape to infinity. We could call such points $z$ where $\mathfrak{Z}(z)$
has fewer than $n-1$ points (counting multiplicities) \emph{singular}. 
It is easy to see that $z$ is singular if and only if the degree of 
$\mathrm{k}(\cdot,z)$ is $<n-1$, and that this is so precisely when
$P_{n-1}(z)=0$. So \emph{the singular points form the zero locus of} $P_{n-1}$. 
Moreover, we know from \cite{HW-ONP} and \cite{HedenmalmPotential} that under
suitable additional topological and regularity conditions, the zero locus is contained 
in the interior of $\calS$ with some control on the distance to the boundary.
A numerical study of the zero set in the case of the kernel for the elliptic Ginibre ensemble
is illustrated in Figure~\ref{fig:Berezin-zeros}.

\subsection{The Berezin density at infinity}

The \emph{orthogonal polynomial wave function}, given by 
$|P_{\n,n}|^2\e^{-2\n Q}$, may be interpreted as an instance of the 
Berezin density for a polynomial ensemble. To see this, we introduce the 
notation $\mathrm{K}(z,w)={\rm K}_{\n,n}(z,w)$ for the weighted kernel
\[
{\rm K}(z,w)={\rm k}(z,w)\,\e^{-\n(Q(z)+Q(w))},
\]
where $\mathrm{k}(z,w)={\rm k}_{\n,n}(z,w)$ is the polynomial 
Bergman kernel \eqref{eq:kernel-def},
and define the Berezin (probability) density $B(z,w)=B_{\n,n}(z,w)$ by
the formula
\[
B(z,w)
=\frac{|{\rm K}_{\n,n}(z,w)|^2}{{\rm K}_{\n,n}(z,z)}=
\frac{|{\rm k}_{\n,n}(z,w)|^2}{{\rm k}_{\n,n}(z,z)}\e^{-2\n Q(w)}.
\]
A limit case of the Berezin density is
\[
|P_{n-1}(w)|^2\e^{-2\n Q(w)}=B(\infty,w)\coloneqq
\lim_{|z|\to+\infty}B(z,w),
\]
where the expression on the left-hand side is the density which
occurs in our potential problem \eqref{eq:pot-exact-2} with $n$ replaced by $n-1$. 
The algorithm presented in this work supplies an asymptotic expansion
for  the density $B_{\n,n}(\infty,\cdot)=|P_{\n,n}|^2\e^{-2\n Q}$ 
and the associated potential $\Pot$, where the two are
connected via \eqref{eq:pot-exact-2} and \eqref{eq:prob2b}.
It is a natural question to ask whether this asymptotic analysis extends beyond
the instance $z=\infty$. We reflect on this in the subsection below.

\subsection{The Berezin density at a general point}
We fix a general point $z\in\C$, and introduce the \emph{Berezin potential} 
$\PotB(z,\cdot)=\PotB_{\n,n}(z,\cdot)$ given by 
\[
\PotB_{\n,n}(z,w)=\int_\C B_{\n,n}(z,\xi)\log|w-\xi|^2\diffA(\xi).
\]
Then $\Pot=\PotB_{\n,n}(\infty,\cdot)$, and the $\partial_w$-derivative of the
Berezin potential equals
\begin{equation}
\label{eq:PotB21}
\partial_w\PotB(z,w)=\int_\C \frac{B_{\n,n}(z,\xi)}{w-\xi}\diffA(\xi).
\end{equation}
By the reproducing
property of kernels, we find that
\begin{multline}
\label{eq:prop101}
\partial_w\PotB(z,w)=\int_\C \frac{B_{\n,n}(z,\xi)}{w-\xi}\diffA(\xi)=
\frac{1}{\mathrm{k}(z,z)}
\int_\C\mathrm{k}(z,\xi)\,\frac{\mathrm{k}(\xi,z)}{w-\xi}\,\e^{-2mQ(\xi)}\diffA(\xi)
\\
=\frac{1}{\mathrm{k}(z,z)}
\int_\C\mathrm{k}(z,\xi)\,\frac{\mathrm{k}(\xi,z)-\mathrm{k}(w,z)}{w-\xi}\,
\e^{-2mQ(\xi)}\diffA(\xi)
+\frac{\mathrm{k}(w,z)}{\mathrm{k}(z,z)}
\int_\C\frac{\mathrm{k}(z,\xi)}{w-\xi}\,
\e^{-2mQ(\xi)}\diffA(\xi)
\\
=\frac{1}{\mathrm{k}(z,z)}
\,\frac{\mathrm{k}(z,z)-\mathrm{k}(w,z)}{w-z}\,
+\frac{\mathrm{k}(w,z)}{\mathrm{k}(z,z)}
\int_\C\frac{\mathrm{k}(z,\xi)}{w-\xi}\,
\e^{-2mQ(\xi)}\diffA(\xi).
\end{multline}
It is our task here to formulate an analogue to Problem 1 which characterizes the Berezin
potential. As a first step, we introduce appropriate notation, and write 
$\mathfrak{c}(z,w):=(w-z)^{-1}$ for the \emph{Cauchy kernel}. 
It follows from \eqref{eq:prop101} that
\begin{equation}
\label{eq:prop102}
\frac{\mathrm{k}(z,z)}{\mathrm{k}(w,z)}\big(\partial_w\PotB(z,w)-\mathfrak{c}(z,w)\big)
+\mathfrak{c}(z,w)=
\int_\C\frac{\mathrm{k}(z,\xi)}{w-\xi}\,\e^{-2mQ(\xi)}\diffA(\xi)
\end{equation}
which expression defines a smooth function of $w\in\C$. This suggests the following 
formulation of the potential problem for the Berezin density: 
\begin{equation}
\label{eq:pot-exact-2B}
\begin{cases}
\vDelta\PotA(z,\cdot)=|q|^2\e^{-2m Q}-\delta_z,\\
\bar\partial q\equiv 0\,\,\,\text{and}\,\,\,q(z)>0,\\
q^{-1}\partial\, \PotA(z,\cdot)\in C^2(\C\setminus\{z\}),
\end{cases}
\end{equation}
coupled with the asymptotic behavior
\begin{equation}
\big(\PotA(z,w),q(w)\big)=
\big(\Ordo(|w|^{-1}),\eta_0
w^{n-1}(1+\Ordo(|w|^{-1})\big)\quad\text{as}\,\,\,|w|\to+\infty,
\label{eq:prob2B}
\end{equation}
for some complex parameter $\eta_0\ne0$.

\begin{thm}
\label{thm:Berezin}
The solution to the system \eqref{eq:pot-exact-2B}--\eqref{eq:prob2B} exists 
uniquely provided that $z$ is nonsingular, i.e, $P_{n-1}(z)\ne0$, and is given by 
\[
(\PotA(z,w),q(w))=\big(\PotB(z,w)-\log|z-w|^2,k(z,z)^{-\frac12}k(w,z)\big).
\]
\end{thm}

The proof of this theorem is deferred to \S\ref{sec:Berezinapp}.
The theorem holds for any point $z\in\C$. However, for an off-spectral point 
$z\in\C\setminus\calS$, things are sufficiently similar to the case $z=\infty$ 
worked out here so that we may guess a shape for $(\PotA(z,\cdot),q)$  and
get the algorithm to supply asymptotic expansions.
We do not carry out any of the necessary details here.
We could mention that the analysis should share some features with 
the method introduced in \cite{HW-OFF}, 
where an asymptotic expansion formula for the kernel 
$\mathrm{k}(w,z)$ was found for a fixed off-spectral point $z$.

\subsection{The Berezin density at bulk points}
The asymptotic expansion formula for the polynomial Bergman kernel 
$\mathrm{k}(z,w)=\mathrm{k}_{\n,n}(z,w)$ for bulk points $z$ is more classical.
We recall that in the context, $z\in\C$ is a \emph{bulk point} if $z$ is in the interior 
of the droplet $\calS$, and in addition, $\varDelta Q(z)>0$ holds at the point.
The H\"ormander $\dbar$-theory allows for localization of the polynomial Bergman
kernel $\mathrm{k}(z,\cdot)$ at bulk points, which gives that the kernel is asymptotically 
the same as the Bergman kernel without polynomial growth restraint \cite{ahm1}.
This would not immediately supply the Berezin potential $\PotB(z,\cdot)$ for bulk points $z$,
but it is expected such a theory could be developed.
We hope to be able to pursue this matter elsewhere. 

\begin{rem}There are of course points $z$ which 
are neither off-spectral nor bulk. In principle there are two possible types of such points. 
We have the \emph{boundary points} $z\in\calS$ as well as the interior points $z$ of the 
droplet $\calS$ where $\vDelta Q(z)=0$. We might refer to the latter as 
\emph{weak} bulk points. Under some radiality assumptions, 
such weak points (or \say{bulk singularities}) were studied
by Ameur-Seo in \cite{AmeurSeo}, who found that Mittag-Leffler kernels
appear in the local scaling limit.
There is currently no method available to expand polynomial Bergman
kernels at general weak bulk points, but we expect that 
such a theory may be developed, at least in particular instances, depending on the type 
of zero of $\vDelta Q$. As for the expansion at regular boundary points, partial insight was
reached in \cite{HW-ONP}, but we could ask for asymptotic expansion there as well. 
\end{rem}

\subsection{ The Ward identity in terms of the Berezin potential}

We aim to characterize the information content of the Ward identities in terms 
of a pointwise identity involving the Berezin potential $\PotB(z,\cdot)$.
To this end, we let $\circledQ=\circledQ_{\n,n}$
denote the function
\[
\circledQ(z)=\tau\int_\C u_n(w)\log|z-w|^2\,\diffA(w)+c,
\] 
where as usual $\tau=n/m$. Here, $u_n=u_{\n,n}$ is the so-called one-point function
 \begin{equation}
\label{eq:onepoint}
u_n(z)=\frac{1}{n}\,\mathrm{K}(z,z)=\frac{1}{n}\,\mathrm{k}(z,z)\,\e^{-2mQ(z)},
\end{equation}
which is a probability density in $\C$. The constant $c=c_{Q,\tau}$ should
only depend on $Q$ and $\tau$, and it  
is chosen such that $\lim_{m\to+\infty}\circledQ=\check{Q}$ holds.
The Ward identity $\mathbb{E} W_{n}^+[v]=0$ of Proposition 3.1 in \cite{ahm3}
has an equivalent formulation in terms of the Berezin potential.

\begin{thm}
\label{thm:Ward}
The Ward identity is equivalent to the near-diagonal behavior
\[
\PotB(z,w)-\PotB(z,z) 
+ \Lambda(z)-\Lambda(w)=\Ordo(|z-w|^2),
\]
where $\PotB=\PotB_{\n,n}$ and we let $\Lambda=\Lambda_{\n,n}$ denote the function 
\[
\Lambda(z)=2\n ({\circledQ}(z) -Q(z))
- \log {\rm K}_{\n,n}(z,z).
\]
\end{thm}

We defer the proof of this statement to later work \cite{HW-Ward}, 
where we intend to explore
higher order Ward identities in the planar context as well.

\begin{rem}
The statements taken from \cite{ahm3} were formulated for $\tau=1$, i.e., $m=n$.
To arrive at the more general statement we should replace $Q$ by $\tau^{-1}Q$
in the statements imported from \cite{ahm3}. 
\end{rem}

\begin{rem}
An interpretation of Theorem \ref{thm:Ward} is that the Ward identity gives local
perturbation information on the Berezin potential $\PotB(z,w)$. However, the theorem 
does not offer any way of calculating the Berezin potential, so that type of information 
should be considered as deeper lying than the Ward identity itself.
\end{rem}

\section{Proofs of Proposition~\ref{lem:exact} and Theorem~\ref{thm:Berezin}}

\label{s:pf-potential-characterization}
\subsection{The proof of Proposition~\ref{lem:exact}}
\label{s:exact}
We verify that if $(\PotV,\pi)$ is a solution
to Problem~\ref{problem:1}, then
$\pi$ is the $n$-th degree monic orthogonal polynomial.
That $\pi$ is a monic polynomial of degree $n$ follows by a Liouville theorem type
argument. We argue that $\PotV$ takes the form
\begin{equation}\label{eq:potV-struct}
\PotV(z)=\int_{\C}|\pi(w)|^2\e^{-2\n Q(w)}\log|z-w|^2\diffA(w),\qquad z\in\C.
\end{equation}
This follows from the fact that 
\begin{equation}
\label{eq:logexpansion}
\log|z-w|^2=\log|z|^2+\log\Big|1-\frac{w}{z}\Big|^2=
\log|z|^2-2\Re\frac{w}{z}+\Ordo\bigg(\frac{|w|^2}{|z|^2}\bigg),
\end{equation}
as $|z|\to \infty$ while $|w|=\ordo(|z|)$,
together with the rapid decay of $|\pi|^2\e^{-2\n Q}$ at infinity and a Liouville-type
argument. We observe that the representation \eqref{eq:potV-struct} implies the
asymptotics
\begin{equation}
\label{eq:appendixU0}
\PotV(z)=\lVert \pi\rVert_{2\n Q}^2\log|z|^2 +\Ordo(|z|^{-1})
\end{equation}
as $|z|\to+\infty$, which proves that $\kappa^2=\lVert \pi\rVert_{2\n Q}^2$.
We may differentiate the relation \eqref{eq:potV-struct} to obtain
\begin{equation}
\label{eq:PotV-diff}
\partial \,\PotV(z)=\int_{\C}\frac{1}{z-w}|\pi(w)|^2\e^{-2\n Q(w)}\diffA(w),\qquad z\in\C,
\end{equation}
which implies the asymptotic relation
\begin{equation}
\label{eq:diff-PotV-asymp}
\partial \,\PotV(z)=\frac{\kappa^2+\Ordo(|z|^{-1})}{z}\qquad\text{as} \quad|z|\to+\infty.
\end{equation}
We proceed to verify that the orthogonality relation
\begin{equation}
\label{eq:orth-pi}
\int_\C q(z)\overline{\pi(z)}\,\e^{-2\n Q(z)}\diffA(z)
=0
\end{equation}
holds for any polynomial $q$ of degree at most $n-1$. 
By assumption, $\partial\PotV$ has the form
\[
\partial\,\PotV=\pi\,\Phi,
\]
where $\Phi$ is $C^2$-smooth. It follows that 
\[
|\pi|^2\e^{-2\n Q}=\vDelta\,\PotV=\bar\partial \partial\,\PotV
=\bar\partial\big(\pi\Phi\big)=\pi\bar\partial\Phi,
\]
from which we conclude that
$\bar\partial\Phi=\overline{\pi}\e^{-2\n Q}$.
We may also conclude that 
\[
\Phi(z)=\frac{1}{\pi(z)}\partial\,\PotV(z)
=\frac{\kappa^2+\Ordo(|z|^{-1})}{z^{n+1}}\qquad\text{as}\quad|z|\to+\infty.
\]
By a standard argument (see e.g.\ \S1.3 in \cite{HedenmalmPotential}), this implies the 
desired orthogonality relation \eqref{eq:orth-pi}.

Conversely, we let $\pi$ be the monic orthogonal polynomial of degree $n$
in the space $L^2(\C,\e^{-2\n Q}\diffA)$, and define $\PotV$ according to \eqref{eq:potV-struct}.
By repeating the above argument, it follows that $\PotV$ has the claimed asymptotic behavior
at infinity, and it only remains to show that 
\[
\pi^{-1}\partial\,\PotV\in C^2(\C).
\]
To this end, we observe that the formula \eqref{eq:PotV-diff} for $\partial\,\PotV$
remains valid, and that
\[
\frac{|\pi(w)|^2}{z-w}=\overline{\pi(w)}\,\frac{\pi(w)-\pi(z)}{z-w} + \frac{\pi(z)\overline{\pi(w)}}{z-w}.
\]
As a consequence, we obtain
\begin{multline}
\partial\,\PotV(z)=\int_{\C}\frac{|\pi(w)|^2}{z-w}\e^{-2\n Q(w)}\diffA(w)
=\int_{\C}\overline{\pi(w)}\,\frac{\pi(w)-\pi(z)}{z-w}\e^{-2\n Q(w)}\diffA(w)
\\+\pi(z)\int_{\C}\frac{\overline{\pi(w)}}{z-w}\e^{-2\n Q(w)}\diffA(w)=
\pi(z)\int_{\C}\frac{\overline{\pi(w)}}{z-w}\e^{-2\n Q(w)}\diffA(w),
\end{multline}
where we use that $(\pi(z)-\pi(w))/(z-w)$ is a polynomial of 
degree at most $n-1$ in $w$.
It follows that
\[
\pi(z)^{-1}\partial\,\PotV(z)=\int_{\C}\frac{\overline{\pi(w)}}{z-w}\e^{-2\n Q(w)}\diffA(w),
\]
which expresses a $C^2$--smooth function by elliptic regularity.\qed

\subsection{The proof of Theorem~\ref{thm:Berezin}}
\label{sec:Berezinapp}
We first check that the proposed solution 
\[
(\PotA(z,\cdot),q)=\big(\PotB(z,\cdot)-\mathfrak{l}(z,\cdot),k(z,z)^{-\frac12}k(\cdot,z)\big)
\]
indeed solves the system \eqref{eq:pot-exact-2B}--\eqref{eq:prob2B}, where 
$\mathfrak{l}(z,w):=\log|z-w|^2$.
From the definition of the Berezin potential $\PotB(z,\cdot)$, it is clear that
\begin{multline}
\vDelta \PotA(z,\cdot)=\vDelta(\PotB(z,\cdot)-\mathfrak{l}(z,\cdot))
=\vDelta\PotB(z,\cdot)-\delta_z=B_{\n,n}(z,\cdot)-\delta_z
\\
=\frac{|\mathrm{k}(z,\cdot)|^2}{\mathrm{k}(z,z)}\,\e^{-2mQ}-\delta_z
=|q|^2\,\e^{-2mQ}-\delta_z,
\end{multline}
where in the last step, we used the proposed formula for $q$. This settles the first line of 
\eqref{eq:pot-exact-2B}. Next, since $z$ is nonsingular, i.e., $P_{n-1}(z)\ne0$,
the proposed formula for $q$ is a polynomial of 
degree precisely $n-1$ with point value $q(z)=k(z,z)^{\frac12}>0$, which settles the second line of 
\eqref{eq:pot-exact-2B} as well as the asymptotic growth of $q$ in \eqref{eq:prob2B}. 
By the argument which leads up to \eqref{eq:appendixU0}, we obtain that
\begin{multline}
\PotB(z,w)= \int_\C\frac{|\mathrm{k}(z,\xi)|^2}{\mathrm{k}(z,z)}\,\e^{-2mQ(\xi)}\log|w-\xi|^2\diffA(\xi)
\\
=\frac{\|\mathrm{k}(z,\cdot)\|^2_{2mQ}}{k(z,z)}\log|w|^2+\Ordo(|w|^{-1})
=\log|w|^2+\Ordo(|w|^{-1})
\end{multline}
as $|w|\to+\infty$. Next, since $\log|w|^2-\log|z-w|^2=\Ordo(|w|^{-1})$, we obtain the required
asymptotic decay of $\PotA(z,\cdot)$ at infinity. 
Finally, in view of the identity \eqref{eq:prop102}, the third line of \eqref{eq:pot-exact-2B} 
follows as well.

It remains to check that the solution to the system \eqref{eq:pot-exact-2B}--\eqref{eq:prob2B} 
is unique. We observe that according to the second line of \eqref{eq:pot-exact-2B}, 
$q$ is entire, and according to the growth bound in \eqref{eq:prob2B}, $q$ must be a 
polynomial of degree equal to $n-1$, as the leading coefficient equals $\eta_0\ne0$. Next, we put
\begin{equation}
\label{eq:circledA}
\PotA^\circledast(z,w):=\int_\C|q(\xi)|^2\e^{-2mQ(\xi)}\log|w-\xi|^2\diffA(\xi)-\log|z-w|^2,
\end{equation}
and, by the argument which gives \eqref{eq:appendixU0}, we find that
\[
\PotA^\circledast(z,w):=\|q\|^2_{2mQ}\log|w|^2-\log|z-w|^2
=(\|q\|^2_{2mQ}-1)\log|w|^2+\Ordo(|w|^{-1}),
\]
as $|w|\to+\infty$. A calculation gives that
\[
\vDelta\PotA^\circledast(z,\cdot)=|q|^2\e^{-2mQ}-\delta_z,
\]
which is the same as $\vDelta\PotA(z,\cdot)$. As a consequence, 
$H:=\PotA(z,\cdot)-\PotA^\circledast(z,\cdot)$ is harmonic, and given the decay of 
$\PotA(z,\cdot)$ required by \eqref{eq:prob2B}, it follows from the above bound that
\[
H(w)=(1-\|q\|^2_{2mQ})\log|w|^2+\Ordo(|w|^{-1})
\]
holds as $|w|\to+\infty$. A Liouville-type argument for harmonic functions now applies to 
show that such weak growth forces $H$ to be constant, and in a second step, that 
$\|q\|_{2mQ}=1$ and $H(w)\equiv0$. We conclude that $\PotA(z,w)=\PotA^\circledast(z,w)$,
which allows us to differentiate \eqref{eq:circledA} to obtain
\begin{equation}
\label{eq:circledA2}
\partial_w\PotA(z,w)=\partial_w\PotA^\circledast(z,w):=
\int_\C \frac{|q(\xi)|^2\e^{-2mQ(\xi)}}{w-\xi}\diffA(\xi)-\frac{1}{w-z}=\Ordo(|w|^{-2}).
\end{equation}
Let us write
\[
\mathscr{X}(z,\cdot)=q^{-1}\partial\PotA(z,\cdot),
\]
which by assumption is $C^2$-smooth in the whole punctured plane $\C\setminus\{z\}$. 
Away from the zeros of $q$, we have 
\begin{multline}
\bar\partial\mathscr{X}(z,\cdot)=
\bar\partial q^{-1}\partial\PotA(z,\cdot)=q^{-1}\bar\partial\partial\PotA(z,\cdot)
q^{-1}\vDelta\PotA(z,\cdot)
\\
=q^{-1}(|q|^2\e^{-2mQ}-\delta_z)=\bar q\e^{-2mQ}
-q(z)^{-1}\delta_z,
\end{multline}
and since $\mathscr{X}(z,\cdot)$ is $C^2$-smooth in $\C\setminus\{z\}$ according to
the third line of \eqref{eq:pot-exact-2B}, and $q(z)>0$ holds at the point, this equality 
holds in the sense of distribution theory in the whole plane. Next, since $q$ has degree
precisely equal to $n-1$, we have the decay
\begin{equation}
\label{eq:decayX}
\mathscr{X}(z,w):=\Ordo(|w|^{-n-1})\quad\text{as}\,\,\,|w|\to+\infty.
\end{equation}
If we put 
\begin{equation}
\mathscr{X}^\circledast(z,w):=
\int_\C\frac{\bar q(\xi)\,\e^{-2mQ(\xi)}}{w-\xi}\,\diffA(\xi)
-\frac{q(z)^{-1}}{w-z},
\end{equation}
an elementary estimates gives that $\mathscr{X}^\circledast(z,w)=\Ordo(|w|^{-1})$
as $|w|\to+\infty$. But then the entire function 
$\mathscr{X}(z,\cdot)-\mathscr{X}^\circledast(z,\cdot)$ has to decay as well, which is only 
possible if it vanishes. So $\mathscr{X}^\circledast(z,\cdot)=\mathscr{X}(z,\cdot)$ as functions. 
By finite geometric series expansion of the Cauchy kernel $\mathfrak{c}(z,w)=(w-z)^{-1}$,
as in \S 1.3 in \cite{HedenmalmPotential}, it is possible to show that the decay 
\eqref{eq:decayX} entails that
\[
\int_\C \xi^j q(z)\overline{q(\xi)}\,\e^{-2mQ(\xi)}\diffA(\xi)=z^j,\qquad j=0,\ldots,n-1.
\]
This is the reproducing property in the space of polynomials of degree $\le n-1$, which
determines the function $q(z)\overline{q(\xi)}$ uniquely: 
\[
q(z)\overline{q(\xi)}=\mathrm{k}(z,\xi).
\]
Then $q(z)=\mathrm{k}(z,z)^{\frac12}$ and $q(w)=\mathrm{k}(z,z)^{-\frac12}\mathrm{k}(w,z)$,
as claimed. In view of this, the formula giving $\PotA(z,w)$ is immediate from \eqref{eq:circledA}.
\qed

\section{An approximate potential problem for orthogonal polynomials}
\label{s:approximate}

\subsection{Heuristic derivation of an approximate potential problem}
\label{s:heuristic}
To reduce the quantity of indices, we usually drop indication of the
parameters $\n, n$ and $\tau=\n/n$. 
Hence, we write in the sequel 
$\calS$, $\Omega$, $\Omega_\epsilon$, $\Gamma$, $P$, and $\hatR$.

We return to the potential problem for the orthogonal polynomials.
We will work with an approximate version of the renormalized version of
Problem~\ref{problem:1},
which is more robust. In particular,
the delicate divisibility condition $P^{-1}\partial\,\Pot\in C^2$
will be replaced by a decay condition in the interior
direction of the droplet, which allows us to introduce 
a cut-off function at little cost and ignore the
interior completely.

We will also work with the $L^2$-normalized orthogonal polynomial rather than the monic
polynomial $\pi=\pi_{m,n}$, and we denote it by $P=P_{m,n}$. 
By now it is understood that 
(see e.g.\ \cite{HW-ONP}, \cite{HedenmalmPotential}), to leading order
\[
P_{m,n}\sim c_{m,n}\phi^n \e^{m\mathscr{Q}}F,
\]
where $c_{m,n}$ is a constant and where $\mathscr{Q}$ and $F$
are holomorphic and bounded on $\Omega_\epsilon$, and where in addition
$F$ is bounded away from zero. The function 
$\mathscr{Q}=\mathscr{Q}_\tau$
satisfies $\Re\mathscr{Q}_\tau=Q$ on 
$\Gamma=\Gamma_\tau$, so that 
\[
\tau\log|\phi_\tau| + \Re \mathscr{Q}_\tau=\breve{Q}_\tau,
\] 
and $F$ should admit an asymptotic expansion in non-negative powers
of $\n^{-1}$. We may consequently conclude that
\[
|P|^2\e^{-2\n Q}\sim c_{m,n}^2|F|^2\e^{-2\n \hatR^2}
\]
holds near $\Gamma$,
where we recall that $\hatR^2=Q-\breve{Q}$ holds there.
Since $F$ is zero-free and holomorphic on $\Omega_\epsilon$,
it follows that $\log|F|$ is harmonic there, which allows us to write
\[
F=\tilde{c}_{m,n}
\e^{h+\imag \tilde{h}},\qquad\text{where }\;
\tilde{c}_{m,n}=c_{m,n}^{-1}\n^{\frac14},
\]
where $h$ is a real-valued bounded harmonic function, while 
$\tilde{h}$ is the harmonic conjugate of $h$, normalized
to vanish at infinity.
In terms of the function $h$, we expect
\[
|P|^2\e^{-2 m Q}\sim m^{\frac12}\,\e^{2h-2m\hatR^2}
\]
to hold in a suitably approximate sense.
It is also convenient to relax the above renormalized version of 
Problem~\ref{problem:1} somewhat,
by asking that the potential, now denoted $\Potu$, satisfies the
weaker growth condition
\[
\Potu(z)=\log|z|^2 + \Ordo(1)
\]
at infinity. The solution $(\Potu,P)$ is no longer unique, but the non-uniqueness
is only up to an additive constant in the potential, i.e.\ the solution space
is $\{(\Pot+a,P):a\in\R\}$, where $(\Pot,P)$ is the unique solution to the
renormalized version of Problem~\ref{problem:1} expressed in \eqref{eq:pot-exact-2}
and \eqref{eq:prob2b}. Since $\Potu=\Pot + a$, we have
\begin{equation}\label{eq:pot-full}
\vDelta\, \Potu =\vDelta\,\Pot = |P|^2\e^{-2\n Q}\sim
m^{\frac12}
\,\e^{2h-2 m\hatR^2}.
\end{equation}
The freedom of adding a constant to $\Pot$ will allow
us to ensure that $\Potu$ decays rapidly in the interior direction to 
the droplet $\calS$
at $\Gamma$, and this effectively replaces the divisibility condition
$P^{-1}\partial\Pot=P^{-1}\partial\Potu\in C^2(\C)$ (given that we are 
convinced that the zeros are located further into the interior of $\calS$). 
The decay in the interior direction is encoded by the ansatz
\begin{equation}
\Potu\sim H \ErfN 
+\frac{G}{\sqrt{2\pi m}}\ExpN
\label{eq:Ansatz-full}
\end{equation}
where $H$ is harmonic while $G$ is
smooth, both real-valued. The function $H$ should be defined on the whole of
$\Omega_\epsilon=\Omega_{\tau,\epsilon}$ and have the appropriate growth
\begin{equation}
\label{eq:growthH}
H(z)=\log|z|^2+\Ordo(1)
\end{equation}
at infinity, but $G$ only needs to be defined 
in a small neighborhood of $\Gamma$. Away from $\Gamma$, 
it is heavily suppressed
by the Gaussian $\ExpN$ anyway.
The principal properties of these functions
are that they admit asymptotic expansions
\begin{equation}
\label{eq:asymp-HG}
H=\coeffH_0 + \n^{-1}\coeffH_1 + \ldots,\qquad
G= \coeffG_0 + \n^{-1}\coeffG_1 + \ldots
\end{equation}
with harmonic and smooth coefficient functions $\coeffH_j$ and $\coeffG_j$, respectively,
which for fixed $\tau$ do not depend on the parameter $\n$.
We stress that the function $\hatR$ is not yet uniquely determined apart
from in a neighborhood of the curve $\Gamma$, but $V^2$ is supposed
to be positive throughout
$\Omega_\epsilon\setminus\Gamma$.

We proceed to solve the problem
\begin{equation}\label{eq:pot-full-h}
\vDelta\, \Potu \sim
m^{\frac12}
\,\e^{2h-2 m\hatR^2}
\end{equation}
algorithmically, where $\Potu$ takes the form \eqref{eq:Ansatz-full}
while $h$ admits an asymptotic expansion
\begin{equation}
\label{eq:asymp-h}
h=\coeffh_0+\n^{-1}\coeffh_1 + \n^{-2}\coeffh_2+\ldots
\end{equation}
with bounded harmonic coefficient functions $\{\coeffh_j\}_j$ which
do not depend on $\n$.
The point is that if a pair $(\Potu, h)$ can be found such that
\eqref{eq:pot-full} holds up to a small error, then we obtain a
good approximation of the wave function $|P|^2\e^{-2\n Q}$.

We do not impose any convergence properties on the asymptotic
expansions \eqref{eq:asymp-HG} and \eqref{eq:asymp-h}. 
At first, we will focus on algorithmic aspects
and produce formal solutions to \eqref{eq:pot-full-h} at the
level of asymptotic expansions.
Later on, we will control the error terms that are introduced
by suitably truncating the asymptotic expansions.

\subsection{A preliminary version of the main result}
A preliminary version of the main result
reads as follows. We work under the standing assumptions of 
\S \ref{s:pot-theory}.
\begin{thm}
\label{thm:main-prel}
There exist {\rm (i)} harmonic functions $H\sim \coeffH_0+\frac{1}{\n}\coeffH_1+\ldots$ with
growth $H(z)=\log|z|+O(1)$ as $|z|\to+\infty$,
{\rm (ii)} smooth functions $G\sim \coeffG_0 + \frac{1}{m}\coeffG_1+\ldots$ and {\rm (iii)} bounded harmonic
functions $h\sim \coeffh_0+\frac{1}{m}\coeffh_1+\ldots$, 
such that if $\Potu$ is given by
\[
\Potu= H \ErfN 
+\frac{G}{\sqrt{2\pi m}}\ExpN,
\]
it holds that
\[
\vDelta\Potu \sim m^{\frac12}\e^{2h-2\n V^2}.
\]
As a consequence of this approximate potential equation, 
the $L^2$-normalized orthogonal polynomial $P=P_{\n,n}$ satisfies
\[
|P|^2\e^{-2\n Q}\sim\n^{\frac12}e^{2h-2\n V^2},
\]
which in turn leads to an asymptotic expansion for $P$.
\end{thm}

The coefficient functions in the various asymptotic expansions
get uniquely determined by a nonlinear algorithm which takes the form
of a fixed-point equation. The symbol $\sim$ denotes equality of
asymptotic expansions. By suitably truncating the asymptotic expansions, 
we obtain an error term of order $\e^{-c_0 \sqrt{\n}}$ with $c_0>0$, valid in a neighborhood 
of width $\asymp \n^{-\frac14}$ of $\overline{\C\setminus\calS_\tau}$.
A more precise statement of this theorem 
is contained in Theorem~\ref{thm:main}
below.

\section{Derivation of the master equation}
\label{s:comput}
\subsection{The ansatz and its consequences}
We make the ansatz
\begin{equation}
\Potu= H \ErfN 
+\frac{G}{\sqrt{2\pi m}}\ExpN,
\label{eq:Ansatz-full-exact}
\end{equation}
where $G$ and $H$ are as above, with asymptotic expansions
given by \eqref{eq:asymp-HG}. This means that $\Potu$ is no longer
considered as an exact solution to the above Laplace equation
\eqref{eq:pot-full}, but only in an approximate sense.
Taking the $\bar\partial$-derivative of the ansatz, we find that
\begin{align}
\label{eq:d-bar-ansatz}
\bar\partial\, \Potu&=
\bar\partial H\ErfN +
\sqrt{\frac{2\n}{\pi}}\bar\partial \hatR
(H-2\hatR G)\ExpN
+\frac{\bar\partial G}{\sqrt{2\pi\n}}\ExpN\\
&=\sqrt{\frac{2\n}{\pi}}\bar\partial \hatR
E\ExpN
+\bar\partial H\ErfN +
\frac{\bar\partial G}{\sqrt{2\pi\n}}\ExpN,
\end{align}
where we have introduced the simplifying notation
\begin{equation}
\label{eq:def-W}
E=H-2\hatR G.
\end{equation} 
Taking the $\partial$-derivative of the above expression,
we arrive at
\begin{multline}
\label{eq:lap-ansatz}
\vDelta\,\Potu=\partial\bar\partial\,\Potu
=-2^{\frac52}\n^{\frac32}\pi^{-\frac12}\hatR|\partial\hatR|^2 \,E\ExpN
\\
+2^{\frac12}\n^{\frac12}\pi^{-\frac12}\Big\{
\partial\hatR(\bar\partial H
-2\hatR\bar\partial G)+
\partial\big(E\bar\partial\hatR\big)\Big\}\ExpN
+\frac{\vDelta G}{\sqrt{2\pi\n}}\ExpN.
\end{multline}
By the product rule, we have the identity
\[
-2\hatR \partial\hatR\bar\partial G
=-2\partial\hatR\bar\partial(G\hatR)
+2|\partial\hatR|^2G,
\]
which gives that
\[
\partial \hatR(\bar\partial H-2\hatR\bar\partial G)=\partial\hatR \bar\partial E 
+ 2|\partial \hatR|^2G.
\]
Combining this with \eqref{eq:lap-ansatz} we find that
\begin{multline}
\label{eq:Lap-U-star}
\vDelta\,\Potu
=-2^{\frac52}\n^{\frac32}\pi^{-\frac12}\hatR|\partial\hatR|^2 \,E\ExpN
\\
+2^{\frac12}\n^{\frac12}\pi^{-\frac12}\Big\{
\partial \hatR\bar\partial E
+2|\partial\hatR|^2G+
\partial E\bar\partial\hatR + E\vDelta \hatR\Big\}\ExpN
+\frac{\vDelta G}{\sqrt{2\pi\n}}\ExpN.
\end{multline}
\subsection{The master equation}
As we want to solve the equation
\begin{equation}
\label{eq:pot-eq-exact}
\vDelta\,\Potu=\n^{\frac12}\e^{2h-2\n \hatR^2},
\end{equation}
we are led to consider
\begin{multline}
\label{eq:lap-eq-final}
\bigg\{-2^{\frac52}\n^{\frac32}\pi^{-\frac12}\hatR|\partial\hatR|^2 \,E
+2^{\frac12}\n^{\frac12}\pi^{-\frac12}\Big(2\Re \{\partial \hatR\bar\partial E\}
+2|\partial \hatR|^2G+ E\vDelta \hatR\Big)\\+\frac{\vDelta G}{\sqrt{2\pi\n}}\bigg\}\ExpN
=m^{\frac12}\,\e^{2h-2\n \hatR^2},
\end{multline}
where we recall that $h$ is supposed to be a bounded harmonic
function with an asymptotic expansion of the form \eqref{eq:asymp-h}.
This should hold in the sense of equalities of formal
asymptotic expansions in powers of $\n^{-1}$,
on a neighborhood
of the outer boundary curve $\Gamma=\Gamma_\tau$.
In terms of the functions $E$, $H$, $G$ and $h$, the equation
\eqref{eq:lap-eq-final} posits that
\begin{equation}
\label{eq:master-full-mod}
-4\n\hatR|\partial\hatR|^2 \,E
+2\Re\big\{\partial\hatR \bar\partial E\big\} 
+E\vDelta \hatR +2|\partial\hatR|^2 G
+\tfrac{1}{2\n}\vDelta G=
\Big(\frac{\pi}{2}\Big)^{\frac12}\e^{2h}
\end{equation}
should hold in the sense of equality of asymptotic expansions.
We recall that we require that $h$ is a real-valued bounded harmonic function
on the fattened exterior domain $\Omega_\epsilon$.

\section{Solving the master equation}
\label{s:sol-master}
In this section (and in the sequel) we will use the notation
$\s=\n^{-1}$. While we will specialize to $\s>0$ in the end, 
it is useful to allow the parameter $\s$ to take on complex values. Restated 
in terms of $\s$, the master equation \eqref{eq:master-full-mod} becomes
\begin{equation}
\label{eq:master-full-mod-theta}
-4\s^{-1}\hatR|\partial\hatR|^2 \,E
+2\Re\big\{\partial\hatR \bar\partial E\big\} 
+ E\vDelta \hatR +2|\partial\hatR|^2 G
+\frac{\s}{2}\vDelta G=
\Big(\frac{\pi}{2}\Big)^{\frac12}\e^{2h}.
\end{equation}

\subsection{Reduction of the master equation: Step 1}
We assume that the equation \eqref{eq:master-full-mod-theta} admits a formal
solution tuple $(E,G,h)$, where $E$ takes the form $E=\coeffE_0+\s\coeffE_1+\ldots$ and where 
$(G,H,h)$ are given by \eqref{eq:asymp-HG}--\eqref{eq:asymp-h} with $\n$ replaced by $\s^{-1}$. 
If we multiply \eqref{eq:master-full-mod} by $\s$ and evaluate at $\s=0$, we obtain
\[
4\hatR|\partial\hatR|^2 \,\coeffE_0=0.
\]
Given that this relation needs to hold in a neighborhood of the curve
$\Gamma$ and since $|\partial\hatR|$ does not vanish there, we find that
the leading order coefficient 
$\coeffE_0$ of $E$ must
vanish identically in a neighborhood of $\Gamma$:
\begin{equation}
\label{eq:E0-eq}
\coeffE_0=0.
\end{equation}
It is thus convenient to work with the function $\E:=\s^{-1}E$, where $\E$
should admit an asymptotic expansion
\begin{equation}
\label{eq:asymp-calE2}
\E=\coeffcalE_0+\s\coeffcalE_1+\s^2\coeffcalE_2+\ldots,
\end{equation}
rather than with $E$. If we rewrite \eqref{eq:master-full-mod-theta} 
as an equation for $(\E,H,G,h)$, we obtain
\begin{equation}
\label{eq:master-full-mod2}
-4\hatR|\partial\hatR|^2 \,\E
+2\s\Re\big\{\partial\hatR \bar\partial \E\big\} 
+\s\E\vDelta \hatR +2|\partial\hatR|^2 G
+\tfrac{1}{2\n}\vDelta G=
\Big(\frac{\pi}{2}\Big)^{\frac12}\e^{2h}.
\end{equation}

As $\s \E=H-2\hatR G$ and since $V$ vanishes on $\Gamma$, we have 
$(H-\s\E)\vert_\Gamma=0$. Given that we know that $H$ meets the growth
bound \eqref{eq:growthH}, we obtain the formula
\begin{equation}\label{eq:H-sol}
H=\log|\phi|^2+\s\Pop_{\Omega}[\E].
\end{equation}
Here, we recall that $\phi$ is the 
conformal mapping of $\Omega$ onto the exterior disk 
$\D_\e$, and denote by
\[
\Pop_\Omega[f]=\int_{\Gamma}P_\Omega(z,w)f(w)\diffs(w)
\]
the usual Poisson extension operator for $\Omega$, given
in terms of the Poisson kernel 
$P_{\Omega}$ for the domain $\Omega$.
For sufficiently regular functions $f$, this 
means that $\Pop_\Omega[f]$ is the bounded harmonic extension of
$f$ to $\Omega$, and if $f$ is real-analytic $\Pop_\Omega[f]$ extends
harmonically past $\Gamma$ in the interior direction. If $\epsilon$
is small enough we may assume that $\Pop_\Omega[E]$ extends harmonically
to the whole fattened exterior domain $\Omega_\epsilon$.
The relation $\s\E=H-2\hatR G$ backwards gives that
\begin{equation}
\label{eq:G-sol}
G=\frac{H-\s \E}{2\hatR}=\frac{\s(\Pop_\Omega-\Iop)[\E]+\log|\phi|^2}{2\hatR},
\end{equation}
where we observe that the above ratio is smooth in view of Weierstrass'
preparation theorem. 
The reason is that the numerator is a real-analytic function which
vanishes along $\Gamma$, while $\hatR$ has non-vanishing gradient
along $\Gamma$.
This allows us to get rid of the function $G$ in the 
equation \eqref{eq:master-full-mod}, so that
\begin{multline}
\label{eq:step-reduct-T}
-4\hatR|\partial\hatR|^2 \,\E
+2\s\Re\big(\partial\hatR \bar\partial \E\big)
+\s \E\vDelta \hatR +\frac{|\partial\hatR|^2}{\hatR}\big(\s(\Pop_\Omega-\Iop)\E
+\log|\phi|^2\big)
\\+\frac{\s}{4}\vDelta\frac{\s(\Pop_\Omega-\Iop)\E+\log|\phi|^2}{\hatR}=
\Big(\frac{\pi}{2}\Big)^{\frac12}\e^{2h}.
\end{multline}
\subsection{Reduction of the master equation: Step 2}
In terms of the linear operator
\begin{equation}
\label{eq:Sop-def}
\Sop [f]=\Sop_\s[f]\eqd \partial\hatR \bar\partial f+\bar\partial\hatR\partial f
+f\vDelta \hatR +|\partial\hatR|^2\frac{(\Pop_\Omega-\Iop)f}{\hatR}
+\frac{\s}{4}\vDelta\frac{(\Pop_\Omega-\Iop)f}{\hatR},
\end{equation}
and the expression
\begin{equation}
\label{eq:L-def}
L\eqd
2|\partial \hatR|^2\frac{\log|\phi|}{\hatR} + \frac{\s}{2}\vDelta
\frac{\log|\phi|}{\hatR},
\end{equation}
the equation \eqref{eq:step-reduct-T} reads, given that $\E$ is real-valued, 
\begin{equation}
\label{eq:reduct-Top2}
-4 \hatR|\partial\hatR|^2 \E+\s\Sop[\E]+L=\Big(\frac{\pi}{2}\Big)^{\frac12}\e^{2h}.
\end{equation}
Hence, since $V$ vanishes on $\Gamma$, the equation for $\E$ and $h$ becomes
\[
L+\s\Sop [\E] = \Big(\frac{\pi}{2}\Big)^{\frac12}\e^{2h}
\qquad\text{on }\;\,\Gamma,
\]
which is equivalent to the equation
\[
\log(L+\s\Sop[\E]) = 2h + \frac{1}{2}\log\frac{\pi}{2}   \qquad\text{on }\;\,\Gamma
\]
so that $h$ is determined from $E$ by
\begin{equation}
\label{eq:h-sol}
h=\frac12\Pop_\Omega\big[\log(L+\s\Sop [\E])\big]-\frac{1}{4}\log\frac{\pi}{2},
\end{equation}
provided that $L+\s\Sop[\E]>0$. We recall $\E$ should admit a full asymptotic expansion
\eqref{eq:asymp-calE2} with leading order behavior $\E=\coeffcalE_0+\Ordo(\s)$,
and it is clear that $L>0$ holds on $\Gamma$ when $|\s|$
is sufficiently small. Hence, 
$L+\s\Sop[\E]>0$ will be satisfied on $\Gamma$ 
for sufficiently small $|\s|$.
In view of the relation \eqref{eq:h-sol}, we 
may rid the master equation of the dependence on $h$ as well, 
and we obtain the following non-linear equation for $\E$:
\begin{equation}
\label{eq:master-simple-W}
-4 \hatR|\partial \hatR|^2 \E + \Big\{L+\s\Sop[\E] - 
\exp\Big(\Pop_\Omega\big[\log(L+\s\Sop [\E])\big]\Big)\Big\}=0.
\end{equation}
The equation \eqref{eq:master-simple-W} may equivalently be stated as
\begin{equation}
\label{eq:master-E'}
\E = \frac{1}{4\hatR|\partial\hatR|^2}
\Big\{L+\s\,\Sop[\E] -\exp\Big(\Pop_\Omega\big[\log(L+\s\,\Sop [\E])\big]\Big)\Big\}
\end{equation}
where we note that the division by $\hatR$ does not produce any
singularity, as the expression in brackets vanishes along $\Gamma$.
This is, as before, related with the Weierstrass division theorem
(see, e.g., Section~6.1 in \cite{HW-ONP} and Section 5.5 in \cite{HedenmalmPotential}).
Moreover, the division by $|\partial \hatR|^2$ is harmless, since $\partial V\neq 0$
in a neighborhood of the curve $\Gamma$.

\subsection{Solution in terms of an asymptotic expansion}
\label{ss:sol-asymp}
We proceed to outline an algorithm which will determine a formal solution $\E$ 
to \eqref{eq:master-E'} in the form of an asymptotic expansion \eqref{eq:asymp-calE2} in $\s$.
In the proofs below, we will take an alternative approach to deriving approximate solutions.
Here, we focus on explaining the ideas, and suppress the technical details for the reader's convenience.

\subsubsection*{The term $\coeffcalE_0$}
The main term $\coeffcalE_0$ is obtained by evaluating both sides of \eqref{eq:master-E'}
at $\s=0$. This gives
\begin{equation}
\label{eq:calE0-def}
\coeffcalE_0=\E\big\vert_{\s=0}=\frac{1}{4\hatR|\partial\hatR|^2}
\Big\{\coeffL_0 - \exp\big(\Pop_\Omega\big[\log \coeffL_0\big]\big)\Big\},
\end{equation}
where we use the notation $\coeffL_0=L\vert_{\s=0}$. 
Since $\coeffL_0-\exp(\Pop_\Omega[\log \coeffL_0])$ vanishes along $\Gamma$
and, the
division by $\hatR$ does not produce any singularity, and, since $\partial\hatR\ne 0$ on $\Gamma$ the term
$\coeffcalE_0$ becomes real-analytic on some neighborhood of $\Gamma$.

\subsubsection*{The term $\coeffcalE_1$} In order to obtain 
the coefficient function $\coeffcalE_1$, we apply
Taylor expansion in $\s$ in equation \eqref{eq:master-E'}. 
In particular, the expansion of the left-hand side is
\begin{equation}
\label{eq:calE0-def}
\E=\coeffcalE_0+\s\coeffcalE_1+\Ordo(|\s|^2),
\end{equation}
where, for formal Taylor series $f$ and $g$, 
the notation $f=g+\Ordo(\s^{k+1})$ means that the Taylor coefficients of $f$ and $g$
agree up to order $k$.
If we introduce the notation $L=\coeffL_0+\s \coeffL_1$ and $\Sop=\coeffSop_0+\s\coeffSop_1$ 
where $\coeffL_j$ and $\coeffSop_j$ do not depend on $\s$ (cf.\ \eqref{eq:Sop-def} and \eqref{eq:L-def}), 
the expression in brackets on the right-hand side of \eqref{eq:master-E'} becomes
\begin{multline}
L+\s\,\Sop[\E] -\exp\Big(\Pop_\Omega\big[\log(L+\s\,\Sop [\E])\big]\Big)
\\
=\coeffL_0+\s (\coeffL_1 + \Sop[\E]\big) 
-\e^{\Pop_\Omega \log \coeffL_0}\exp\big(\Pop_\Omega\log\big(1+\s\tfrac{\coeffL_1+\Sop[\E]}{\coeffL_0}\big)\big)\\
=\coeffL_0-\e^{\Pop_\Omega \log \coeffL_0}+\s (\coeffL_1 + \coeffSop_0[\E]+\s\coeffSop_1[\E]\big) 
\\-\e^{\Pop_\Omega \log \coeffL_0}\Big[\exp\big(\Pop_\Omega\log\big(1+\s\tfrac{\coeffL_1+\coeffSop_0[\E]
+\s\coeffSop_1[\E]}{\coeffL_0}\big)\big)-1\Big].
\end{multline}
Inserting the expansion \eqref{eq:asymp-calE2} of $\E$ and Taylor expanding the logarithm to first order, 
we obtain the Taylor expansion
\begin{multline}
L+\s\,\Sop[\E] -\exp\Big(\Pop_\Omega\big[\log(L+\s\,\Sop [\E])\big]\Big)
\\
=\coeffL_0-\e^{\Pop_\Omega \log \coeffL_0}+\s \Big(\coeffL_1 + \coeffSop_0[\coeffcalE_0]
-\e^{\Pop_\Omega \log \coeffL_0}\Pop_\Omega\Big[\tfrac{\coeffL_1+\coeffSop_0[\coeffcalE_0]}{\coeffL_0}\Big]\Big)
+\Ordo(|\s|^2).
\end{multline}
We hence see that
the the equality \eqref{eq:master-E'} up to terms of order 1 in $\s$ amounts to
\begin{equation}
\label{eq:E1-def}
\coeffcalE_1=\frac{1}{4\hatR |\partial\hatR|^2}\Big(\coeffL_1 + \coeffSop_0[\coeffcalE_0]
-\e^{\Pop_\Omega \log \coeffL_0}\Pop_\Omega\Big[\tfrac{\coeffL_1+\coeffSop_0[\coeffcalE_0]}{\coeffL_0}\Big]\Big).
\end{equation}
On the curve $\Gamma$, we have
\[
\coeffL_1 + \coeffSop_0[\coeffcalE_0]-\e^{\Pop_\Omega \log \coeffL_0}\Pop_\Omega
\Big[\tfrac{\coeffL_1+\coeffSop_0[\coeffE_0]}{\coeffL_0}\Big]
=\coeffL_1+\coeffSop_0[\coeffcalE_0]-\e^{\log \coeffL_0}\tfrac{\coeffL_1+\coeffSop_0[\coeffcalE_0]}{\coeffL_0}=0,
\]
which means that we may divide smoothly by $\hatR|\partial\hatR|^2$ in \eqref{eq:E1-def}, 
so $\E_1$ becomes real-analytic 
on a neighborhood of $\Gamma$.

\subsubsection*{The general term} 
We now describe how to obtain the coefficient functions $\coeffcalE_j$
for $j\ge 2$. As we enter the general step in the algorithm, 
we have determined the coefficient functions $\coeffcalE_{0}, \ldots,\coeffcalE_{k}$
for some given positive integer $k$.
In view of \eqref{eq:master-E'}, the next coefficient, $\coeffcalE_{k+1}$, 
must be equal to the $(k+1)$-st Taylor coefficient of
\[
\frac{1}{4\hatR|\partial\hatR|^2}
\Big\{L+\s\,\Sop[\E] -\exp\Big(\Pop_\Omega\big[\log(L+\s\,\Sop [\E])\big]\Big)\Big\}.
\]
We will refrain from writing the explicit formula for $\coeffcalE_{k+1}$, but 
we should explain that it may be expressed in terms of the known coefficient 
functions $\coeffcalE_j$ for $j\le k$, the operators $\Sop$ and $\Pop_\Omega$ 
and the function $L$. 
To see this, we first introduce the notation
\[
\E=\sum_{j=0}^k\s^j \coeffcalE_j + \sum_{j=k+1}^\infty \s^j\coeffcalE_j=:\calE^{\rm I} 
+ \s^{k+1}\calE^{\rm II},
\]
where we stress that $\calE^{{\rm I}}$ is considered \emph{known} at this stage,
and find that
\begin{multline}
L+\s\,\Sop[\E] -\exp\Big(\Pop_\Omega\big[\log(L+\s\,\Sop [\E])\big]\Big)\\
=L+\s\,\Sop[\E] -\e^{\Pop_\Omega\log \coeffL_0}\exp\Big(\Pop_\Omega\big[\log(1+\s\,\Sop [\E]/\coeffL_0)\big]\Big)
\\=L+\s\,\Sop[\E] -\e^{\Pop_\Omega\log \coeffL_0}
+\e^{\Pop_\Omega\log \coeffL_0}\exp\Big(\Pop_\Omega\big[\log\big(1+\,\s(\Sop [\E^{\rm I}]+
\s^{k+2}\Sop[\E^{\rm II}])/\coeffL_0\big)\big]\Big).
\end{multline}
Next, we rewrite the expression in the exponent on the right-hand side as
\begin{multline}
\Pop_\Omega\big[\log(1+\,\big(\s\Sop [\E^{\rm I}]+
\s^{k+2}\Sop[\E^{\rm II}]\big)/\coeffL_0)\big]
\\
=\Pop_\Omega\Big[\log\big(1+\,\s\Sop [\E^{\rm I}]/\coeffL_0\big)+
\log\Big(1+\s^{k+2}\frac{\Sop[\E^{\rm II}]}{\coeffL_0+\s\Sop[\E^{\rm I}]}\Big)\Big]
\\
=\Pop_\Omega\Big[\log\big(1+\,\s\Sop [\E^{\rm I}]/\coeffL_0\big)\Big]+\Ordo(\s^{k+2}),
\end{multline}
where the last equality holds since $\coeffL_0$ is non-zero on $\Gamma$ 
for sufficiently small $\s$. As a consequence, we obtain that
\begin{multline}
\label{eq:Ck-pre}
L+\s\,\Sop[\E] -\exp\Big(\Pop_\Omega\big[\log(L+\s\,\Sop [\E])\big]\Big)\\
=L+\s\Sop[\calE^{\rm I}]-
\e^{\Pop_\Omega\log \coeffL_0}\exp\Big(\Pop_\Omega\big[\log(1+\s\,
\Sop [\calE^{\rm I}]/\coeffL_0)\big]\Big)
+\Ordo(\s^{k+2}),
\end{multline}
which shows that the $(k+1)$-st Taylor coefficient of the right-hand side of the
master equation \eqref{eq:master-E'} 
only depends on $\calE^{\rm I}$; and not on 
the higher order terms contained in $\calE^{\rm II}$.
Combining \eqref{eq:master-E'} with \eqref{eq:Ck-pre}, we obtain the structural
formula
\begin{equation}
\label{eq:general-term-k+1}
\E_{k+1}=\frac{1}{4\hatR|\partial\hatR|^2}\mathfrak{F}_k(\E_0,\ldots,\E_k)
\end{equation}
where the expression $\mathfrak{F}_k(\E_0,\ldots,\E_k)$ can be computed in terms of known functions.
As in the previous steps, the fact that 
\[
L+\s\,\Sop[\E] -\exp\Big(\Pop_\Omega\big[\log(L+\s\,\Sop [\E])\big]\Big)=0\quad\text{on }\,\Gamma
\]
ensures $\mathfrak{F}_{k}$ vanishes on $\Gamma$, and no singularity is created 
when we divide by $\hatR|\partial\hatR|^2$ in \eqref{eq:general-term-k+1} with $j=k+1$.
Hence, $\coeffcalE_{k+1}$ is real-analytic in a neighborhood of $\Gamma$.

\bigskip

By induction, the above algorithm produces a sequence $\coeffcalE_0,\coeffcalE_1,\coeffcalE_2,\ldots$ of coefficient
functions, so that $\E=\coeffcalE_0+\s\coeffcalE_1+\s^2\coeffcalE_2 + \ldots$ supplies a formal
solution to \eqref{eq:master-E'}.

\subsection{The asymptotic expansion of the coefficient functions}
\label{s:approx-coeff}
We recall the equations \eqref{eq:H-sol}, \eqref{eq:G-sol}
and \eqref{eq:h-sol} for the remaining coefficient functions involved
in the approximate potential equation \eqref{eq:master-full-mod}, 
which express $H$, $G$ and $h$ in terms of
$E=\s\E$. With $\E$ denoting the formal asymptotic obtained
expansion in the previous section, we let $\E^\approx$ 
denote a suitable approximation. For instance, we may choose
to truncate the series at a level $k=k_\s$
\[
\E^{\approx}=\coeffcalE_0+\s\coeffcalE_1+\ldots + \s^k\coeffcalE_k.
\]
It turns out to be convenient for estimates to do things somewhat differently
and only truncate at the end, but the approximation $\E^\approx$ will
have a growing number of coefficients equal to $\coeffcalE_j$; 
we discuss this more in the following section.
Given $\E^\approx$, we put
\begin{equation}
\label{eq:def-H-H-approx}
H^{\approx}=\log|\phi|^2+\s\Pop_\Omega[\E^{\approx}],
\qquad G^{\approx}=\frac{H^{\approx}-\s \E^{\approx}}{2\hatR}
\end{equation}
and finally
\begin{equation}
\label{eq:def-h-approx}
h^{\approx}=\Pop_\Omega\log\big(L+\s\Sop[\E^{\approx}])
-\frac{1}{4}\log\frac{\pi}{2}.
\end{equation}
The function $h^\approx$ admits an asymptotic expansion $h^\approx=\coeffh_0^\approx+\s\coeffh_1^\approx+\ldots$,
and, provided that the approximate solution $\coeffcalE$ has correct coefficients up to order $k$, 
we have $\coeffh^\approx_j=\coeffh_j$ for $j\le k+1$. The leading coefficient function becomes
\[
\coeffh_0=\Pop_\Omega\log(\coeffL_0)-\frac14\log\frac{\pi}{2}
\]
and, by Taylor expanding the logarithm, we see that the subsequent corrections are given by
\[
\coeffh_1=\Pop_\Omega\Big\{\frac{\coeffL_1+\coeffSop_0[\coeffcalE_0]}{\coeffL_0}\Big\}
\]
and
\[
\coeffh_2=\Pop_\Omega\Big\{\frac{\coeffSop_1[\coeffcalE_0]+
\coeffSop_0[\coeffcalE_1]}{\coeffL_0}-\frac12\coeffh_1^2\Big\},
\]
where $\coeffcalE_0$ and $\coeffcalE_1$ are given by \eqref{eq:calE0-def} and \eqref{eq:E1-def}, respectively.
Here, we recall that $L=\coeffL_0+\s\coeffL_0$ 
and $\Sop=\coeffSop_0+\s\coeffSop_1$, where the coefficients $\coeffL_j$ and $\coeffSop_j$
do not depend on $\s$.

In the end, we will further truncate $h^{\approx}$ by replacing it with its Taylor polynomial
$\mathfrak{h}_{k}$ in $\s$, for a suitable choice of $k$.
With these functions in hand, we will construct 
a suitably approximate local potential for $\n^{\frac12}\e^{2h^\approx-2\n \hatR^2}$ from the ansatz
\eqref{eq:Ansatz-full},
and derive the asymptotics of the orthogonal polynomial $P=P_{\n,n}$ from this fact; see 
Proposition~\ref{prop:approx-Pot-eq} and Theorem~\ref{thm:main}.

\subsection{The key estimates}
The equation \eqref{eq:master-E'} may equivalently 
be stated as the fixed point equation
\begin{equation}
\label{eq:Top-fixpt}
\Tope \E=\E,
\end{equation} 
where $\Tope$ denotes the nonlinear operator
\begin{equation}\label{eq:Top-def-prel}
\Tope[f]=\Tope_\s[f]\eqd \frac{1}{4\hatR|\partial\hatR|^2}
\Big\{L+\s\,\Sop[f] - 
\exp\Big(\Pop_\Omega\big[\log(L+\s\,\Sop [f])\big]\Big)\Big\},
\end{equation}
where $\Sop$ and $L$ were defined in \eqref{eq:Sop-def} 
and \eqref{eq:L-def}, respectively. We remark that the expression in brackets
vanishes on $\Gamma$, so the division by $\hatR|\partial V|^2$ does not create any
singularity.

Rather than truncating the asymptotic expansion of $\E$, we will
obtain the approximate solution $\E^{\approx}$ as an iterate $\E^{\approx}=\Tope^{k+1}[0]$
of the nonlinear operator $\Tope$. 
While the operator $\Sop$ preserves real-analyticity in some neighborhood
of $\Gamma$, its iterates can in fact grow quite wildly. 
Hence, we need to carefully choose the number $k=k_\s$ of iterates of $\Tope$
in relation to the parameter $\s$, and only keep iterating as long as the size of 
$\Tope^{k+1}[0]$ remains suitably controlled. 

Below, we will prove estimates which imply the following lemmas. 
We recall that while in the end we will apply the results with $\s=\frac1\n$,
it is convenient to allow $\s$ to take complex values.

\begin{lem}
\label{lem:approx-fix-pt-prel}
There exist positive constants $C_0$, $c_0$, $\varepsilon$ and $\varrho$ and a neighborhood
$\mathcal{N}$ of $\Gamma$, all depending only on $\hatR$ and $\phi$, such that
if $k=k_\s$ is the integer part of $\varepsilon |\s|^{-\frac12}$, 
the function $\E^\approx=\Tope^{k+1}[0]$ satisfies
\begin{equation}
\label{eq:Ek-main}
\sup_{z\in\mathcal{N}}
\big|\E^\approx(z)-\Tope[\E^\approx](z)\big|
\le C_0\exp\big(-c_0|\s|^{-\frac12}\big)
\end{equation}
for $\s\in\D(0,\varrho)$.
\end{lem}

Recall the definitions of $H^\approx$, $G^\approx$
and $h^\approx$ from \eqref{eq:def-H-H-approx}--\eqref{eq:def-h-approx}.
\begin{lem}
\label{lem:bdd-coeff-prel}
With $\mathcal{N}$ and $\varrho$ as in Lemma~\ref{lem:approx-fix-pt-prel},
the functions $\E^\approx$, $H^\approx$, $G^\approx$ are real analytic on $\mathcal{N}$,
while $h^\approx$ is bounded and harmonic on $\mathcal{N}\cup\Omega$ for $\s\in\D(0,\varrho)$.
In addition, the $C^1(\mathcal{N})$--norms of the 
functions $\E^\approx$, $H^\approx$, $G^\approx$, $h^\approx$ and the harmonic conjugate
$(h^\approx)^*$ of $h^\approx$ in $\Omega$
are uniformly bounded for $\s\in\D(0,\varrho)$.
\end{lem}

In order to obtain an asymptotic expansions for the orthogonal polynomial in the end
(rather than a formula in terms of iterates of $\Tope$),
we need to know that $h^\approx$ is well approximated by its $k$-th Taylor polynomial in
the variable $\s$ with $k=k_\s$, 
which we denote by $\mathfrak{h}^\approx=\mathfrak{h}^\approx_k$.

\begin{lem}
\label{lem:truncation-h-prel}
With $\mathcal{N}$ and $\varrho$ as in Lemma~\ref{lem:approx-fix-pt-prel} and $k=k_\s$,
there exist positive constants $C_1$ and $c_1$ such that 
\[
\sup_{z\in\mathcal{N}\cup\Omega}\big|\mathfrak h^\approx-h^\approx\big|
\le C_1\exp(-c_1|\s|^{-\frac12})
\]
for $\s\in\D(0,\varrho)$.
\end{lem}

For the reader's convenience, we defer the rather technical proofs of these 
three lemmas to \S\ref{s:est-Top}-\ref{s:truncation} below. 
In fact, we will obtain stronger versions which control
the relevant quantities in a scale of Banach spaces which 
implies uniform pointwise and $C^1$-control
on $\mathcal{N}$; see Theorem~\ref{thm:main-iteration}, Proposition~\ref{prop:bdd-coeff} and 
Lemma~\ref{lem:truncation}, respectively. In the upcoming section, we proceed 
to investigate the consequence of these results for 
the potential equation and for the asymptotics of the orthogonal
polynomials.
\section{Deriving the asymptotic expansion}
\label{s:proof}

\subsection{Cut-off functions}
We introduce a cut-off function $\chi_1$ 
with the following properties. The function $\chi_1$ is $C^\infty$--smooth
and takes values in $[0,1]$. It
vanishes on a neighborhood of $\C\setminus \mathcal{N}$ while
it equals one on a smaller ring domain
$\mathcal{N}'$ containing $\Gamma$. Without loss of generality we assume that
$\mathcal{N}=\phi^{-1}(\A)$ and $\mathcal{N}'=\phi^{-1}(\A')$, where $\A$ and $\A'$
are two centered open annuli containing the unit circle with $\overline{\A'}\subset\A$.
A way to construct $\chi_1$ is to begin with a $C^\infty$--smooth radial 
cut-off function $\chi_1'$ with values
in $[0,1]$ which
equals 1 on the annulus $\mathbb{A}'$ and vanishes
off the larger annulus $\mathbb{A}$,  
and put $\chi_1\eqd \chi_1'\circ\phi$.
Later on, we will also need a $C^\infty$--smooth cut-off function $\chi_2:\C\to[0,1]$ 
which vanishes on $\C\setminus(\Omega\cup \mathcal{N})$
and equals $1$ on $\Omega\cup\mathcal{N}'$. We construct this function in a similar fashion
based on a radial cut-off function.
In addition, we require that $\chi_2=\chi_1$ holds on
the set $\mathcal{N}'\cup (\C\setminus\Omega)$, so that the support of the 
difference $\chi_2-\chi_1$ is contained in $\Omega\setminus\mathcal{N}'$.

\subsection{The approximate potential equation}
\label{s:approx-pot}
We now return to the original problem, and specify $\s$ to $\s=\n^{-1}$,
where $\n$ is the large semiclassical weight parameter.
By a slight abuse of notation, we will write $k_\n$ for $k_\s$. That is,
$k_\n$ is the integer part of $\varepsilon\sqrt{\n}$ for some fixed $\varepsilon>0$.
We let
\begin{equation}
\label{eq:Pot-u-approx-def}
\Potu^\approx=\chi_1\Big(H^\approx\ErfN + \frac{G^\approx}{\sqrt{2\n\pi}}\ExpN\Big)
+(\chi_2-\chi_1)H^\approx,
\end{equation}
where $\E^\approx=\Tope^{k_\n+1}[0]$ and where the auxiliary coefficient
functions $H^\approx$ and $G^\approx$ are given by \eqref{eq:def-H-H-approx}.
We also recall the harmonic function $h^\approx$ given by \eqref{eq:def-h-approx}.
The following is an immediate consequence of Lemma~\ref{lem:approx-fix-pt-prel} 
and these definitions.

\begin{prop}
\label{prop:approx-Pot-eq}
We have 
\[
\vDelta\, \Potu^{\approx}=\chi_1\sqrt{\n}\e^{2h^\approx-2\n\hatR^2}
\Big(1+\Ordo\big(\e^{-c_0\sqrt{\n}}\big)\Big) 
+ \Ordo\big(1_{\mathcal{N}\setminus\mathcal{N}'}\ExpN\big)
\]
as $\n\to\infty$, where the implicit 
constants are uniform throughout the plane, and where $c_0>0$
is the positive constant from Lemma~\ref{lem:approx-fix-pt-prel}.
\end{prop}

\begin{proof}
We introduce the notation $\Potu^{\approx}_*$ for the approximate potential
\[
\Potu^{\approx}_*=H^{\approx}\ErfN + \frac{G^{\approx}}{\sqrt{2\n\pi}}\ExpN,
\]
which may possibly be defined only on 
the neighborhood $\mathcal{N}$ of $\Gamma$.
Then the potential $\Potu^{\approx}$ satisfies $\Potu^{\approx}=\chi_1\Potu^{\approx}_* 
+ (\chi_2-\chi_1)H^{\approx}$, 
and we may calculate the Laplacian
of $\Potu^{\approx}$ as 
\begin{multline}
\vDelta\,\Potu^{\approx}=\chi_1\vDelta\,\Potu^{\approx}_*
+ 2\Re\big(\bar\partial \chi_1\;\partial\,\Potu^{\approx}_* 
+ \bar\partial(\chi_2-\chi_1)\;\partial H^{\approx}\big)
\\+\vDelta\,\chi_1\;\Potu^{\approx}_* + \vDelta (\chi_2-\chi_1)\; H^{\approx}
=:\chi_1\vDelta\,\Potu^{\approx}_* + W,
\end{multline}
where we treat $W=W_\n$ as part of the error term.

We first claim that the term $W$ is of order 
$\Ordo\big(1_{\mathcal{N}\setminus\mathcal{N}'}\ExpN\big)$,
so that it fits within the given error parameters.
Indeed, it vanishes outside $\mathcal{N}\setminus\mathcal{N}'$, and on
the set $\mathcal{N}\setminus(\mathcal{N}'\cup \Omega)$, 
we use the fact that $\chi_1=\chi_2$ there, so that
\[
W=2\Re\big(\bar\partial \chi_1\,\partial\,
\Potu^{\approx}_*\big)+\vDelta\,\chi_1\,\Potu^{\approx}_*\qquad
\text{on}\quad \mathcal{N}\setminus(\mathcal{N}'\cup \Omega).
\]
We observe that $\Potu^{\approx}_*$
and $\partial\,\Potu^{\approx}_*$ are both of order 
$\Ordo(\ExpN)$ on $\mathcal{N}\setminus\mathcal{N}'$.
This holds in view of the decay estimate \eqref{eq:erf-est1} 
for the error function together with Lemma~\ref{lem:bdd-coeff-prel}, 
which asserts that the coefficient functions $H^{\approx}$ and $G^{\approx}$ 
together with their gradients are uniformly bounded on $\mathcal{N}$. 
As a consequence, we obtain
\[
W=\Ordo\big(1_{\mathcal{N}\setminus\mathcal{N}'}\ExpN\big)
\qquad\text{on}\quad \C\setminus\Omega.
\]
On the set $(\mathcal{N}\setminus\mathcal{N}') \cap \Omega$,
we instead have $\nabla \chi_2\equiv0$, and hence
\[
W=2\Re\big\{\bar\partial \chi_1\,\partial(\Potu^{\approx}_*-H^\approx)\big\}
+\vDelta\,\chi_1\,(\Potu^{\approx}_*-H^{\approx})\qquad\text{on}\quad\Omega.
\]
But in view of the estimate \eqref{eq:erf-est2} and the fact that
\[
\Potu^{\approx}_*-H^{\approx}=H^{\approx}(\ErfN-1) + \frac{1}{\sqrt{2\pi\n}}G^{\approx}\ExpN,
\]
we find that both $\Potu^{\approx}_*-H^{\approx}$ and $\nabla (\Potu^{\approx}_*-H^\approx)$ are of order
$\Ordo(\ExpN)$ on the open set $(\mathcal{N}\setminus\mathcal{N}')\cap \Omega$, so it follows that
\[
W=\Ordo(1_{\mathcal{N}\setminus\mathcal{N}'}\ExpN)
\]
globally. 

It only remains to show that
\begin{equation}
\label{eq:main-term-pot-eq-truncated}
\chi_1\vDelta\,\Potu^{\approx}_*=\chi_1 \n^{\frac12}\,\e^{2h^\approx-2\n\hatR^2}
\Big(1+\Ordo\big(\e^{-c_0\sqrt{\n}}\big)\Big)
\end{equation}
where the implicit constant is supposed to be uniform on $\mathcal{N}$.
To see this, we should relate \eqref{eq:main-term-pot-eq-truncated} to the nonlinear operator $\Tope$
studied above. In view of the equations \eqref{eq:Lap-U-star},
\eqref{eq:lap-eq-final}, and \eqref{eq:master-full-mod}, with $H$ and $G$ replaced by $H^{\approx}$ and $G^{\approx}$,
respectively, we have that
\[
\vDelta \,\Potu^{\approx}_* - \n^{\frac12}\e^{2h^\approx-2\n\hatR^2}
=\Big(\frac{\pi}{2}\Big)^{-\frac12}\n^{\frac12}A\,\e^{-2\n \hatR^2},
\]
where
\begin{multline}
A=-4\hatR|\partial\hatR|^2 \,\E^\approx
+\tfrac{2}{\n}\Re\big\{\partial\hatR \bar\partial \E^\approx\big\} 
+\tfrac{1}{\n}\E^\approx
\vDelta \hatR +2|\partial\hatR|^2 G^\approx
\\+\tfrac{1}{2\n}\vDelta G^\approx-
\Big(\frac{\pi}{2}\Big)^{\frac12}\e^{2h^\approx}.
\end{multline}
Following the calculations in \S\ref{s:sol-master}, in particular the equations
\eqref{eq:step-reduct-T}, \eqref{eq:Sop-def}, \eqref{eq:reduct-Top2}, \eqref{eq:h-sol} 
and \eqref{eq:master-simple-W} (cf. also the definition of $\Tope$ in \eqref{eq:Top-def-prel}), 
we find that $A$
may be expressed as
\[
A=-4\hatR|\partial\hatR|^2\Big(\E^\approx-\Tope[\E^\approx]\Big).
\]
As a consequence, we obtain
\begin{multline}
\vDelta\, \Potu^\approx_*-\n^{\frac12}\e^{2h^\approx-2\n\hatR^2}=
-4\Big(\frac{\pi}{2}\Big)^{-\frac12}\n^{\frac12}\hatR|\partial\hatR|^2\big(\E^\approx-\Tope[\E^\approx]\big)\ExpN
\\=\Ordo\big(\n^{\frac12}\e^{-c_0\sqrt{\n}}\ExpN\big),
\end{multline}
where the last step follows from Lemma~\ref{lem:approx-fix-pt-prel}.
But this bound is equivalent to the equation \eqref{eq:main-term-pot-eq-truncated},
and hence the proof is complete.
\end{proof}

\subsection{Asymptotics of orthogonal polynomials}
In \S\ref{s:approx-pot}, we proved that high iterates $\calE^\approx=\Tope^{k+1}[0]$
supply approximate solutions to the fixed-point equation $\Tope\calE\approx \calE$. If we form the
functions $H^\approx$ and $G^\approx$ as in \S\ref{s:approx-coeff}
and the potential $\Potu$ according to the ansatz \eqref{eq:Ansatz-full}, then $\vDelta \Potu$
should approximate to the orthogonal polynomial wave function $|P|^2\e^{-2\n Q}$.
To actually prove that it is correct, we need
go back to $P=P_{\n,n}$ itself, which amounts to a kind of polarization.
As such, the result is essentially equivalent
to the main result of \cite{HedenmalmPotential}, but with a different
proof and a new formula for $P_{\n,n}$. Another difference is that
we do not work with truncated asymptotic expansions, but rather the
full iterates $\calE^\approx=\Tope^{k+1}[0]$. These iterates admit asymptotic
expansions in negative powers of $\n$, but 
will in general contain terms of all orders.
In view of Lemma~\ref{lem:truncation-h-prel} (see also \S\ref{s:truncation} below), 
it is also possible to work with finite asymptotic 
expansions with an $\n$-dependent number of terms.

We let $h^\approx$ be as above, 
and denote by $(h^\approx)^*$ the harmonic conjugate of
$h^\approx$ in the domain $\Omega$, normalized to vanish at infinity. 
We recall that $\tau$ is a fixed parameter with $0<\tau\le 1$ 
such that the assumptions on $\calS=\calS_\tau$
in \S\ref{s:assumptions} hold. 

\begin{thm}
\label{thm:main}
Let $c_0>0$ be the positive constant from 
Lemma~\ref{lem:approx-fix-pt-prel}. Then, we have
\[
\big\lVert P_{\n,n}-\chi_2\n^{\frac14}\phi^n\,
\e^{h^\approx+\imag (h^\approx)^*+\n \mathcal{Q}} \big\rVert_{2\n Q}
=\Ordo\big(\e^{-c_0\sqrt{\n}}\big),
\]
as $\n,n\to+\infty$ with $n=\tau\n$.
\end{thm}

Theorem~\ref{thm:main} will be shown to follow from Proposition~\ref{prop:approx-Pot-eq}.
Several of the necessary steps have appeared previously in the works 
\cite{HW-ONP, HW-ONP2, HedenmalmPotential}, and for this reason we keep the
presentation brief.

\begin{proof}[Proof]
We let $f$ be the holomorphic function on $\Omega$ defined by
\[
f(z)=\n^{\frac14}\phi^n\e^{h^\approx+\imag (h^\approx)^*+\n\mathcal{Q}}.
\]
Notice that $|f(z)|\asymp |z|^n$ at infinity, and that $f$ extends holomorphically
to $\Omega\cup\mathcal{N}$.
As a first step, we apply a version of H{\"o}rmander's 
$L^2$-estimates for the $\bar\partial$-operator to find a polynomial $p$
of degree $n$
such that
\[
\int_{\C}\big|p-\chi_2 f\big|^2\e^{-2\n Q}\diffA(z)=\Ordo\big(\e^{-2d\n}\big).
\]
where $d>0$, see e.g.\ \S3.6 in \cite{HW-ONP}. It remains to show that $p$
is close to the sought orthogonal polynomial $P_{\n,n}$.

We will use the approximate potential 
equation from Proposition~\ref{prop:approx-Pot-eq}
together with the structure of $\Potu^{\approx}$ and
Green's formula to verify that $p$ is 
appropriately approximately orthogonal to any polynomial $q$ of
degree at most $n-1$. If $q$ is such a polynomial, we observe that
\[
\int_{\C} q\,\overline{p}\,\e^{-2\n Q}\diffA=
\int_{\C} \chi_2\,q\,\overline{f}\,\e^{-2\n Q}\diffA 
+ \Ordo(\e^{-d\n}\lVert q\rVert_{2\n Q}),
\]
so it remains to estimate the first term on the right-hand side.
But we have
\[
\int_{\C} \chi_2 q\,\overline{f}\,\e^{-2\n Q}\diffA 
=\int_{\C} \chi_2\, \frac{q}{f}\,|f|^2\,\e^{-2\n Q}\diffA 
=\n^{\frac12}\int_{\C} \chi_2\,\frac{q}{f}\, \e^{2 h^\approx - 2\n\hatR^2}\diffA.
\]
We rewrite this expression in the form
\[
\int_{\C} \chi_2\,\frac{q}{f}\, \e^{2 h^\approx - 2\n\hatR^2}\diffA
=\int_{\C} \chi_1\,\frac{q}{f}\, \e^{2 h^\approx - 2\n\hatR^2}\diffA
+\n^{-\frac12}\int_{\C} (\chi_2-\chi_1) q\,\overline{f}\,\e^{-2\n Q}\diffA,
\]
where we should stress that the integrand in the last term on the right-hand side 
is non-zero only on $\Omega\setminus\mathcal{N}'$. 
Moreover, since $h^\approx$ is uniformly bounded on $\Omega$, we have the pointwise estimate
\[
|f|\,\e^{-\n Q}\le C\n^{\frac14}\e^{-\n(Q-\breve{Q})} \qquad\text{on}\quad\Omega,
\] 
and a simple
application of the Bernstein-Walsh lemma (see e.g.\ Proposition~2.3 in \cite{HW-ONP})
shows that $|q|\e^{-\n Q}\le C\n^{\frac12}\e^{-\n(Q-\breve{Q})}\lVert q\rVert_{2\n Q}$,
for some positive constant $C=C(Q)$.
Hence, we find that (since $\chi_2-\chi_1\ge 0$ automatically)
\begin{multline}
\bigg|\int_{\C} (\chi_2-\chi_1) q\,\overline{f}\,\e^{-2\n Q}\diffA\bigg|
\le C^2\n^{\frac34} \lVert q\rVert_{2\n Q} \int_{\C} (\chi_2-\chi_1)\e^{-2\n (Q-\breve{Q})}\diffA
\\=\Ordo(\e^{-d'\n}\lVert q\rVert_{2\n Q})
\end{multline}
for some $d'>0$, where the last step follows from the fact that
$Q-\breve{Q}$ is bounded below by a positive constant on $\mathrm{supp}(\chi_2-\chi_1)$
and the known growth of $Q-\breve{Q}$ at infinity.

We proceed to show that
\begin{equation}
\label{eq:final-est}
\n^{\frac12}\int_{\C} \chi_1\,\frac{q}{f}\, \e^{2 h^\approx - 2\n\hatR^2}\diffA
=\Ordo\big(\e^{-c_0\sqrt{\n}}\lVert q\rVert_{2\n Q}\big),
\end{equation}
and it is here that the potential equation enters. Indeed, in view of
Proposition~\ref{prop:approx-Pot-eq}, we have that
\[
\n^{\frac12}\chi_1\e^{2 h^\approx - 2\n\hatR^2}=\vDelta\, \Potu^\approx
+ \Ordo\big(\big(\chi_1\n^{\frac12}\e^{-c_0\sqrt{\n}}+1_{\mathcal{N}\setminus\mathcal{N}'}\big)\ExpN\big),
\]
from which it follows that
\begin{multline}
\label{eq:error-q/f-int1}
\n^{\frac12}\int_{\C} \chi_1\,\frac{q}{f}\, \e^{2 h^\approx - 2\n\hatR^2}\diffA
=\int_{\C} \frac{q}{f}\vDelta\,\Potu^\approx
\diffA 
\\+ \Ordo\bigg(\n^{\frac12}\int_{\C} \chi_1 \frac{|q|}{|f|}\e^{-c_0\sqrt{\n}}\e^{-2\n \hatR^2}\diffA
+\int_{\mathcal{N}\setminus\mathcal{N}'}\frac{|q|}{|f|}\e^{-2\n \hatR^2}\diffA\bigg).
\end{multline}
In addition, in view of the uniform boundedness of $h^\approx$
on the support of $\chi_1$, we find that
\begin{multline}
\label{eq:1/f-bd}
|f|^{-1}\e^{-2\n \hatR^2} = \Ordo\big(|f|^{-1}\e^{2 h^\approx - 2\n\hatR^2}\big)
\\
=\Ordo\big(\n^{-\frac12}|f|^{-1}\,|f|^2\e^{ - 2\n Q}\big)
=\Ordo\big(\n^{-\frac12}|f|\e^{- 2\n Q}\big),
\end{multline}
and hence, an application of the Cauchy-Schwartz inequality
gives that 
\begin{multline}
\int_{\C} \chi_1 \frac{|q|}{|f|}\e^{-c_0\sqrt{\n}}\e^{-2\n \hatR^2}\diffA
=\Ordo\Big(\n^{-\frac12}\e^{-c_0\sqrt{\n}}\int_{\C}\chi_1\,|q|\,|f|\,\e^{-2\n Q}\diffA\Big)
\\= \Ordo\big(\n^{-\frac12}\e^{-c_0\sqrt{\n}}\lVert q\rVert_{2\n Q}\lVert \chi_1 f\rVert_{2\n Q}\big).
\end{multline}
Since $\lVert \chi_1 f\rVert_{2\n Q}$ is clearly bounded (this will be verified below),
it follows that the first Ordo expression on the right-hand side of \eqref{eq:error-q/f-int1}
is of order $\Ordo(\e^{-c_0\sqrt{\n}}\lVert q\rVert_{2\n Q})$.

Turning to the last error term in \eqref{eq:error-q/f-int1}, we observe that
\eqref{eq:1/f-bd} implies that
\[
\int_{\mathcal{N}\setminus\mathcal{N}'} \frac{|q|}{|f|}\e^{-2\n \hatR^2}\diffA
\le C\n^{-\frac12}\int_{\mathcal{N}\setminus\mathcal{N}'} |q|\,|f|\,\e^{-2\n Q}\diffA,
\]
so an application of the Cauchy-Schwartz inequality together with the observation that
\[
\int_{\mathcal{N}\setminus\mathcal{N}'} |f|^2\e^{-2\n Q}
\le C\n^\frac12\int_{\mathcal{N}\setminus\mathcal{N}'}\e^{-2\n \hatR^2}\diffA=\Ordo(\e^{-d''\n})
\]
for some $d''>0$ shows that the second error term in \eqref{eq:error-q/f-int1} is 
of order $\Ordo(\e^{-d''\n}\lVert q\rVert_{2\n Q})$. Hence, it follows that
\[
\n^{\frac12}\int_{\C} \chi_1\,\frac{q}{f}\, \e^{2 h^\approx - 2\n\hatR^2}\diffA=
\int_{\C} \frac{q}{f}\,\vDelta\, \Potu^\approx\,\diffA
+\Ordo\big(\e^{-c_0\sqrt{\n}}\lVert q\rVert_{2\n Q}\big)
\]
for $\n$ sufficiently large.
To complete the proof of \eqref{eq:final-est}, we observe that by Green's formula,
\[
\int_{\C} \frac{q}{f}\,\vDelta\, \Potu^\approx\,\diffA=
\lim_{R\to+\infty}\frac{1}{2}\int_{\T(0,R)}
\Big(\frac{q}{f}\partial_{\rm n}\Potu^\approx
-\Potu^\approx \partial_{\rm n}\frac{q}{f}\Big)\diffs=0,
\]
where $\T(0,R)=\{z:|z|=R\}$, and where the last step 
relies on the facts that the gradient of $\Potu^\approx$
decays like $\Ordo(1/|z|)$ at infinity, and that 
$q/f$ decays like $\Ordo(|z|^{-1})$ at infinity.
If we put the pieces together, it follows that
\begin{equation}
\label{eq:approx-orth}
\int_{\C} q\,\overline{p}\,\e^{-2\n Q}\diffA
=\Ordo\big(\e^{-c_0\sqrt{\n}}\lVert q\rVert_{2\n Q}\big),
\end{equation}
as $\n\to+\infty$.

Up to the point where we applied Green's formula, 
essentially the same computations may be repeated
with $q$ replaced by $p$, which is very close to 
$f$ on $\Omega$. Specifically, we have that
\begin{align}
\label{eq:norm-1}
\int_{\C} |p|^2\,\e^{-2\n Q}\diffA&=
\int_{\C} \chi_1 |f|^2\,\e^{-2\n Q}\diffA + \Ordo\big(\e^{-d\n}\big)\\&=
\int_{\C} \vDelta \,\Potu^\approx \diffA + \Ordo\big(\e^{-c_0\sqrt{\n}}\big)\\
&= \lim_{R\to+\infty}\frac12\int_{\T(0,R)}\partial_{\rm n}\Potu^\approx\diffs+\Ordo(\e^{-c_0\sqrt{\n}})\\
&=\lim_{R\to+\infty}\frac12\int_{\T(0,R)}\partial_{\rm n} H^\approx\diffs+\Ordo(\e^{-c_0\sqrt{\n}})
\\ &=1+\Ordo(\e^{-c_0\sqrt{\n}}),
\end{align}
where the last step relies on the fact that 
$\partial_{\rm n}H^\approx(z)=2|z|^{-1}+\Ordo(|z|^{-2})$ at $\T(0,R)$
for large $R$.

As was explained on pp.~349--350 in \cite{HW-ONP}, 
the bounds \eqref{eq:norm-1} and \eqref{eq:approx-orth} together 
with a simple duality argument then yield the claim.
\end{proof}

We let $\tau$ and $c_0$ be as in Theorem~\ref{thm:main}.

\begin{cor}
\label{cor:main}
There exists a positive number $\delta>0$, such that 
for $z\in\C$ such that $\mathrm{dist}_\C(z,\Omega)\le \delta\, \n^{-\frac14}$, we have
the asymptotics
\[
|P_{\n,n}|^2=\sqrt{\n}\,\e^{2 h^\approx +2\n \breve{Q}_\tau}
\Big(1+\Ordo\big(\e^{-c_0\sqrt{\n}}\big)\Big)
\]
as $\n, n\to+\infty$ with $\n=\tau\n\to+\infty$.
\end{cor}
\begin{rem}
{\rm (a)}\; In terms of the \emph{wave function} $|P_{\n,n}|^2\,\e^{-2\n Q}$, we find that
\[
|P_{\n,n}(z)|^2\e^{-2\n Q(z)}=\sqrt{\n}\,\e^{2 h^\approx-2\n \hatR^2}
\Big(1+\Ordo\big(\e^{-c_0\sqrt{\n}}\big)\Big).
\]
{\rm (b)}\;In addition, if $I_0$ is a compact subinterval of $(0,1]$
such that the assumptions in Section~\ref{s:assumptions}
hold for all $\tau\in I_0$, then the implicit constant in Theorem~\ref{thm:main} and
Corollary~\ref{cor:main} remains uniformly bounded, provided that $\tau\in I_0$.
\end{rem}
\begin{proof}
The corollary follows from Theorem~\ref{thm:main} by a 
standard application of the Bernstein-Walsh inequality,
and a comparison between
$\e^{-c_0\sqrt{\n}}\e^{-2\n \hatR^2}$, which is the order of the error term,
and the function $\e^{-2\n(Q-\check{Q})}$ in the interior direction
of the droplet $\calS$ at $\Gamma$. The details are given in
\cite{HedenmalmPotential}, \S4.4.
\end{proof}

\section{Estimates for the nonlinear operator \texorpdfstring{$\Tope$}{}}
\label{s:est-Top}
\subsection{Banach scales of quantitatively real-analytic functions}
In this section and the next, we return to the notation $\s=\n^{-1}$ 
and agree to let $\s$ take complex values.
Recall that the equation \eqref{eq:master-E'} may equivalently 
be stated as the fixed point equation
\begin{equation}
\label{eq:Top-fixpt-rep}
\Tope \E=\E,
\end{equation} 
where $\Tope$ is the nonlinear operator
\begin{equation}\label{eq:Top-def}
\Tope[f]=\Tope_\s[f]\eqd \frac{1}{4\hatR|\partial\hatR|^2}
\Big\{L+\s\,\Sop[f] - 
\exp\Big(\Pop_\Omega\big[\log(L+\s\,\Sop [f])\big]\Big)\Big\},
\end{equation}
where $\Sop$ and $L$ were defined in \eqref{eq:Sop-def} 
and \eqref{eq:L-def}, respectively.
We will obtain approximate solutions $E^\approx$ by iteration of $\Tope$, i.e.
we put $\E^\approx=\E_k$ for a suitable $k=k_\s$, where
\[
\E_k=\Tope^{k+1}[0].
\]
In order to see to what extent this iteration scheme
will supply an approximate solution to \eqref{eq:Top-fixpt}
we will need Lipschitz-type
estimates for the operator $\Tope$. 
More specifically, in view of the calculation
\begin{multline}
\label{eq:Top-diff}
\Tope[f]-\Tope[g]\\
=\frac{1}{4\hatR|\partial\hatR|^2}
\bigg\{\s\,\Sop[f-g] 
+ \exp\Big(\Pop_\Omega\big[\log(L+\s\, \Sop [f])]\Big)
-\exp\Big(\Pop_\Omega\big[\log(L+\s\,\Sop [g])]\Big)\bigg\}
\\
=\frac{1}{4\hatR|\partial\hatR|^2}
\bigg\{\s\,\Sop[f-g] +\e^{\Pop_\Omega[\log L]}\Rop(f,g)\Big\},
\end{multline}
where we have introduced the notation
\begin{equation}
\label{eq:discrepancy-f-g}
\Rop(f,g)\eqd
\exp\Big(\Pop_\Omega\big[\log(1+\s\,\Sop [f]/L)]\Big)
-\exp\Big(\Pop_\Omega\big[\log(1+\s\,\Sop [g]/L)]\Big),
\end{equation}
the non-trivial part is to estimate the nonlinear contribution $\Rop(f,g)$
in terms of the size of $f-g$.
The problem is that the nonlinear operator $\Tope$ 
fails to be Lipschitz in the na{\"i}ve sense, e.g., 
when acting on spaces of real-analytic functions
on a given domain containing $\Gamma$. Instead, we will need to work with a decreasing 
scale of spaces $\mathscr{H}_\sigma^\infty(\Gamma)$ 
of suitably \emph{quantitatively real-analytic} functions
defined in a complexified neighborhood of the boundary curve $\Gamma$,
and we will be able to prove Lipschitz estimates for $\Tope$ 
only if we allow the co-domain to increase along this scale.

We follow the approach touched upon in \cite{HW-ONP} and 
further developed in \cite{HedenmalmPotential}, 
and denote by $\mathbb{A}(\rho)$ the annulus
\[
\A(\rho)=\{z\in\C:\rho\le |z|\le \rho^{-1}\}
\]
where $0<\rho<1$, and for $0<\sigma<1$ we introduce the $2\sigma$-fattened diagonal
annulus in $\C^2$ by
\[
\hat{\A}(\rho,\sigma)=\big\{(z,w)\in\A(\rho)\times \A(\rho): |z-w|\le 2\sigma\big\}.
\] 
It turns out to be especially convenient to
restrict attention to parameters $\rho$ and $\sigma$ such that
\[
\rho=\rho(\sigma)=\frac{1}{\sigma+\sqrt{1+\sigma^2}}=1-\sigma+\Ordo(\sigma^2),
\]
where the indicated error terms refers to the limiting procedure $\sigma\to0$.
We let $\hat{\A}(\sigma)$ denote the annulus $\hat{\A}(\sigma)=\hat{\A}(\rho(\sigma),\sigma)$.

We denote by $\mathscr{H}^\infty_\sigma(\Gamma)$ the space of real-analytic 
functions $f$ on the neighborhood of $\Gamma$ given by $\phi^{-1}(\A(\rho(\sigma)))$, 
such that the Hermitian-analytic 
polarization $f_\phi(z,w)$ of $f\circ\phi^{-1}$
is bounded on $\hat{\A}(\sigma)$. Here, we recall 
that $\phi$ is the exterior conformal mapping
$\phi:\Omega\to\D_\e$, which preserves the point at infinity, and we 
will only work with small enough parameters $0<\sigma<1$ such 
that $\phi^{-1}$ extends conformally to $\A(\rho(\sigma))$.
If we endow $\mathscr{H}^\infty_\sigma(\Gamma)$ with the norm
\[
\lVert f\rVert_\sigma=\sup_{(z,w)\in\hat{\A}(\sigma)}|f_\phi(z,w)|,
\]
it becomes a commutative Banach algebra.
We will show that $\Tope$ maps $\mathscr{H}_\sigma^\infty(\Gamma)$ into
$\mathscr{H}_{\sigma'}^\infty(\Gamma)$ whenever $\sigma>\sigma'$, and 
obtain appropriate Lipschitz bounds of the form
\[
\lVert \Tope(f)-\Tope(g)\rVert_{\sigma'}\le C\Big(\frac{1}{(\sigma')^2(\sigma-\sigma')^2}+
\frac{|\s|}{(\sigma')^2(\sigma-\sigma')^4}\Big)\,
|\s|\,\lVert f-g\rVert_{\sigma},
\]
valid for $f$ and $g$ e.g.\ in a ball of fixed radius $r_0$ in
$\mathscr{H}_\sigma^\infty(\Gamma)$ and for $\s$ in a fixed disk $\D(0,\varrho_0)$.
Since for some $\sigma=\sigma^*$, the initial iterate 
$\E_0=\Tope[0]$ belongs to a fixed ball in $\mathscr{H}_{\sigma^*}^\infty(\Gamma)$, 
this will allow us to keep good control of iterates $\Tope^k[0]$
as long as $k$ is not too large. We will quantify these statements shortly.

\subsection{Basic estimates}
To get started, we need the following Lemma. It 
follows directly from \cite[\S5.5]{HedenmalmPotential}.

\begin{lem}
\label{lem:div}
Assume that $f\in \mathscr{H}_\sigma^\infty(\Gamma)$ 
vanishes along $\Gamma$.
Then $(1-|\phi|^2)^{-1}f\in \mathscr{H}_{\sigma'}^\infty(\Gamma)$ 
for any $\sigma'<\sigma$, and we have
\[
\big\lVert (1-|\phi|^2)^{-1} f\big\rVert_{\sigma'}\le 
6\frac{\lVert f\rVert_{\sigma}}{\sigma-\sigma'}.
\]
\end{lem}

As a first step, we show that the two linear operators $\Pop_\Omega$ and $\Sop$
act boundedly as operators $\mathscr{H}^\infty_\sigma(\Gamma)\to \mathscr{H}^\infty_{\sigma'}(\Gamma)$ 
whenever $\sigma>\sigma'$. In fact, for $\Pop_\Omega$ this holds true for $\sigma=\sigma'$
as well, and we have
\begin{equation}\label{eq:sigma-bound-P}
\lVert \Pop_\Omega\rVert_{\sigma\to\sigma}\eqd
\lVert \Pop_\Omega\rVert_{\mathscr{H}^\infty_\sigma(\Gamma) 
\to \mathscr{H}^\infty_{\sigma}(\Gamma)}\le \frac{6}{\sigma},
\end{equation}
as follows from \cite[\S5.5]{HedenmalmPotential}.
With the simplifying notation
\begin{equation}\label{eq:sigma-bound-S}
\lVert \Sop\rVert_{\sigma\to\sigma'}
\eqd\lVert \Sop\rVert_{\mathscr{H}^\infty_\sigma(\Gamma) 
\to \mathscr{H}^\infty_{\sigma'}(\Gamma)}
\end{equation}
for the operator norm of $\Sop$, we have the following result.

\begin{prop}
\label{prop:S-P-bd}
For $\sigma$ and $\sigma'$ with $0<\sigma'<\sigma<1$, we have that
\begin{multline}
\lVert \Sop\rVert_{\sigma\to\sigma'}\le \lVert \vDelta V\rVert_{\sigma} +
\frac{12\sigma'\lVert \partial\hatR\rVert_{\sigma}\lVert \phi'\rVert_{\sigma}
+42\lVert(\partial V)^{2}\rVert_{\sigma}
\lVert (1-|\phi|^2)/\hatR\rVert_{\sigma}}{(\sigma-\sigma')\sigma'}
\\+ |\s|
\frac{10206 \big\lVert \phi'\big\rVert_{\sigma}^2\,
\big\lVert(1-|\phi|^2)/\hatR\big\rVert_{\sigma}}{(\sigma-\sigma')^3\sigma'}.
\end{multline}
\end{prop}
Ignoring the precise value of the constants, we have
\begin{equation}
\label{eq:bd-Sop-simple}
\lVert \Sop\rVert_{\sigma\to\sigma'}\le M_0(\sigma,\sigma',\s):=C_0\left(
\frac{1}{(\sigma-\sigma')\sigma'}+\frac{|\s|}{(\sigma-\sigma')^3\sigma'}\right)
\end{equation}
where the constant $C_0$ depends only on $V$ and $\phi$, provided that $\sigma$
and $\sigma'$ are small enough (depending only on $V$ and $\phi$).
\begin{proof}
Before proceeding to estimate
the operator norm of $\Sop$, we recall the bounds
\begin{equation}
\label{eq:d-bar-bd}
\big\lVert \bar\partial f\big\rVert_{\sigma'}\le 
\frac{6\lVert f\rVert_\sigma\lVert \phi'\rVert_{\sigma'}}{\sigma-\sigma'}
\end{equation}
and
\begin{equation}
\label{eq:I-P-bd}
\big\lVert (\Iop-\Pop_\Omega)f\big\rVert_{\sigma}\le 
\frac{7\lVert f\rVert_{\sigma}}{\sigma}
\end{equation}
which follow from the equation (5.5.2) in 
\cite{HedenmalmPotential} together with a change of variables, 
and from (5.5.3) in \cite{HedenmalmPotential}, respectively.

The operator $\Sop$ acts by
\[
\Sop(f)=\partial\hatR \bar\partial f+\bar\partial\hatR \partial f
+f\vDelta \hatR +|\partial\hatR|^2\frac{(\Pop_\Omega-\Iop)f}{\hatR}
+\frac{\s}{2}\vDelta\frac{(\Pop_\Omega-\Iop)f}{2\hatR},
\]
so it follows from the triangle inequality and the above bound \eqref{eq:d-bar-bd} 
for the $\bar\partial$-operator that
\begin{multline}
\lVert \Sop\rVert_{\sigma\to\sigma'}\le 
\frac{12\lVert \partial\hatR\rVert_{\sigma'}\lVert \phi'\rVert_{\sigma'}}{\sigma-\sigma'}
+\lVert \vDelta V\rVert_{\sigma} 
+\big\lVert(\partial V)^{2}\big\rVert_{\sigma'}
\sup_{\lVert f\rVert_{\sigma}=1}
\bigg\lVert\frac{(\Iop-\Pop_\Omega)f}{\hatR}\bigg\rVert_{\sigma'}\\
+ \frac{|\s|}{4}\sup_{\lVert f\rVert_{\sigma}= 1}
\bigg\lVert\vDelta\frac{(\Pop_\Omega-\Iop)f}{\hatR}\bigg\rVert_{\sigma'}.
\end{multline}
It remains to estimate the last two terms. The first term may be estimated
using \eqref{eq:I-P-bd} and Lemma~\ref{lem:div}, and we obtain
\[
\sup_{\lVert f\rVert_{\sigma}=1}
\bigg\lVert\frac{(\Iop-\Pop_\Omega)f}{\hatR}\bigg\rVert_{\sigma'}
\le \frac{42\big\lVert V^{-1}(1-|\phi|^2)\big\rVert_{\sigma'}}{(\sigma-\sigma')\sigma}.
\] 
The last remaining term will be 
handled by repeated application of \eqref{eq:d-bar-bd} and \eqref{eq:I-P-bd}.
Namely, we insert two intermediate parameters $\sigma_0=\sigma'+\frac13(\sigma+\sigma')$ and 
$\sigma_1+\frac23(\sigma+\sigma')$ and apply \eqref{eq:d-bar-bd}
to obtain
\[
\bigg\lVert\Delta\frac{(\Iop-\Pop_\Omega)f}{\hatR}\bigg\rVert_{\sigma'}
\le \frac{6\lVert \partial g\rVert_{\sigma_0}\lVert \phi'\rVert_{\sigma'}}{\sigma_0-\sigma'},
\]
where we have introduced the auxiliary function
\[
g=\frac{(\Iop-\Pop_\Omega)}{\hatR}f.
\]
But applying \eqref{eq:d-bar-bd} again (formally a conjugated version)
followed by the Banach algebra property and Lemma~\ref{lem:div}, we find that
\[
\lVert \partial g\rVert_{\sigma_0}\le \frac{6\lVert \phi'\rVert_{\sigma_0}}{(\sigma_1-\sigma_0)}
\bigg\lVert\frac{(\Iop-\Pop_\Omega)f}{\hatR} \bigg\rVert_{\sigma_1}
\le \frac{36\times 7 \lVert \phi'\rVert_{\sigma_0}\big\lVert(1-|\phi|^2)/\hatR\big\rVert_{\sigma_1}}
{(\sigma_1-\sigma_0)(\sigma-\sigma_1)\sigma_1}
\lVert f\rVert_{\sigma}.
\]
Putting these observations together, we arrive at
\begin{multline}
\lVert \Sop\rVert_{\sigma\to \sigma'}\le
\frac{12\sigma'\lVert \partial\hatR\rVert_{\sigma'}\lVert \phi'\rVert_{\sigma'}
+42\lVert(\partial V)^{2}\rVert_{\sigma'}\big\lVert(1-|\phi|^2)/\hatR
\big\rVert_{\sigma'}}{(\sigma-\sigma')\sigma'}
\\+\lVert \vDelta V\rVert_{\sigma} + \frac{|\s|}{4}
\frac{6^3\times 7 \lVert \phi'\rVert_{\sigma_0}\lVert \phi'\rVert_{\sigma'}
\big\lVert(1-|\phi|^2)/\hatR
\big\rVert_{\sigma_1}}{(\sigma_0-\sigma')(\sigma_1-\sigma_0)(\sigma-\sigma_1)\sigma_1}
\end{multline}
The proof follows by calculating the constants and using the
monotonicity of the norm 
$\lVert \cdot\rVert_{\sigma}$ in $\sigma$.
\end{proof}

We record the following simple bound for some quantities related to $\hatR$ and $\phi$
which will appear repeatedly.

\begin{prop}
\label{prop:L-bd}
There exist constants $0<\sigma^*<1$ and $\epsilon_0>0$, depending only on $\hatR$ and $\phi$,
such that all the quantities 
\begin{equation}
\big\lVert (\partial V)^{-2}\big\rVert_{\sigma}, 
\quad \big\lVert V^{-1}(1-|\phi|^2)\big\rVert_{\sigma},\quad 
\big\lVert \phi'\big\rVert_{\sigma},\quad \text{and} 
\quad \big\lVert \partial V\big\rVert_{\sigma},
\end{equation}
as well as
\begin{equation}
\sup_{|\s|\le\epsilon_0}\big\lVert 1/L\big\rVert_{\sigma},\quad\text{ and }\quad  
\sup_{|\s|\le\epsilon_0}\big\lVert \Tope[0]\big\rVert_{\sigma}
\end{equation}
are uniformly bounded for $0<\sigma\le \sigma^*$.
\end{prop}

\begin{proof}
We recall the definitions \eqref{eq:L-def} and $\eqref{eq:Top-def}$ 
of $L$ and $\Tope$, respectively, and write $L=\coeffL_0+\s\coeffL_1$
where the coefficient functions do not depend on $\s$.
We have
\[
\Tope[0]= \frac{1}{4\hatR|\partial\hatR|^2}
\Big\{L - \exp\Big(\Pop_\Omega\big[\log L\big]\Big)\Big\}=
\Big\{\coeffL_0+\s\coeffL_1 - \exp\Big(\Pop_\Omega\big[\log \coeffL_0+\s\coeffL_1\big]\Big)\Big\}.
\]
It is clear that the relevant quantities are bounded for some $\sigma^*$ and for $\s=0$.
Possibly after shrinking $\sigma^*$ slightly, the coefficient $\coeffL_1$ 
is also bounded in $\mathscr{H}^\infty_{\sigma^*}(\Gamma)$. 
Using the fact that both $1/t$ and $\exp(t)$ are Lipschitz (the former away from the origin),
we conclude that for $\s$ sufficiently small, both the quantities 
\[
\sup_{|\s|\le\epsilon_0}\big \lVert 1/L\big \rVert_{\sigma^*},\quad 
\sup_{|\s|\le\epsilon_0}\big\lVert \Tope[0]\big\rVert_{\sigma^*}
\]
are bounded. The claim now follows from the monotonicity 
of $\lVert \cdot\rVert_{\sigma}$ in $\sigma$.
\end{proof}
We fix, once and for all, the radius $r_0$ to be
\begin{equation}
\label{eq:r0}
r_0:=2\sup_{|\s|\le \epsilon_0}\big\lVert \Tope[0]\big\rVert_{\sigma^*}
=2\sup_{|\s|\le \epsilon_0}\big\lVert \calE_0\big\rVert_{\sigma^*}.
\end{equation}

\subsection{A preliminary Lipschitz-type estimate}
We proceed to bound the 
quantity $\Rop(f,g)$ that appeared in \eqref{eq:discrepancy-f-g}.

\begin{prop}
\label{prop:bd-R}
Fix $\sigma$ and $\sigma'$ with $0<\sigma'<\sigma<1$.
There exists a constant $\n_1$, depending only on $L$, $\Omega$, $\sigma$, $\sigma'$ and $r_0$,
such that for $|\s|\le \frac{1}{\n_1}$ and for all $f$ and $g$ in the
ball of radius $r_0$ in $\mathscr{H}^\infty_\sigma(\Gamma)$, we have that
\begin{equation}
\lVert \Rop(f,g)\rVert_{\sigma'}\le
6 |\s|\,\lVert \Pop_\Omega\rVert_{\sigma'\to\sigma'}
\lVert \Sop\rVert_{\sigma\to\sigma'}\lVert 1/L\rVert_{\sigma'}\lVert f-g\rVert_{\sigma}.
\end{equation}
\end{prop}
\begin{rem}
We mention that the constant $\n_1$ is explicitly given below in \eqref{eq:n1}.
\end{rem}
\begin{proof}
In the proof, we keep in mind that if $F$ is holomorphic on 
$\D(0,\lVert f\rVert_{\sigma})$, 
then the polarization of $(F\circ f)_\phi$
equals $F\circ f_\phi$. We also recall that the spaces 
$\mathscr{H}_\sigma^\infty(\Gamma)$ are Banach algebras.

In view of the bound
\[
\big\lVert\s\, \Sop[f]/L\big\rVert_{\sigma'}\le 
|\s|\,\lVert 1/L\rVert_{\sigma'}\lVert \Sop\rVert_{\sigma\to \sigma'}
\lVert f\rVert_{\sigma}.
\]
and the fact that $\big|\log(1+t)\big|\le 2|t|$ holds for complex $t$ 
whenever $|t|\le \frac12$,
we find that for $f$ of norm at most $r_0$ in $\mathscr{H}^\infty_\sigma(\Gamma)$, 
we have that 
\begin{multline}
\big\lVert \Pop_\Omega\log(1+\s\,\Sop[f]/L)\big\rVert_{\sigma'}\le 
\lVert \Pop_\Omega \rVert_{\sigma'\to\sigma'}
\big\lVert \log(1+\s\,\Sop[f]/L)\big\rVert_{\sigma'}
\\
\le 2|\s|\;\lVert \Pop_\Omega \rVert_{\sigma'\to\sigma'} \lVert\Sop\rVert_{\sigma\to \sigma'}
\lVert 1/L\rVert_{\sigma'}
\lVert f\rVert_{\sigma}
\end{multline}
for all $\s$ with $|\s|$ sufficiently small,
and similarly for $g$. Here, \say{sufficiently small} means that
\begin{equation}
\label{eq:n0}
|\s| \le \frac{1}{\n_0},\quad \text{where }\;
\n_0 \eqd \max\big\{\epsilon_0^{-1},4 M_0(\sigma,\sigma',\s)
\sup_{|\s|\le\epsilon_0}\lVert 1/L\rVert_{\sigma'}r_0\big\},
\end{equation}
where $\epsilon_0=\epsilon_0(V,\phi)$ is the positive 
parameter from Proposition~\ref{prop:L-bd}
while $M_0(\sigma,\sigma',\s)$ is the constant from 
\eqref{eq:bd-Sop-simple}.
Since 
\[
|\exp(t)-\exp(t')|\le 3|t-t'|
\] 
whenever $t,t'$ are complex numbers
with $|t|\le 1$ and $|t'|\le 1$, 
we find that for all $f$ and $g$ in the ball of radius $r_0$
in $\mathscr{H}_\sigma^\infty(\Gamma)$
and for all $\s$ such that
\begin{equation}
\label{eq:n1}
|\s|\le \frac{1}{\n_1},\quad\text{where }\; \n_1\eqd \max\Big\{\n_0,
2\lVert \Pop_\Omega \rVert_{\sigma'\to\sigma'} M_0(\sigma,\sigma',\s)
\sup_{|\s|\le\epsilon_0}\lVert 1/L\rVert_{\sigma'}r_0\Big\},
\end{equation}
we have that
\begin{multline}
\lVert \Rop(f,g)\rVert_{\sigma'}\le 3 
\big\lVert \Pop_\Omega\log(1+\s\,L\Sop[f]/L)-
\Pop_\Omega\log(1+\s\,\Sop[g]/L)\big\rVert_{\sigma'}\\
\le 3\lVert\Pop_\Omega\rVert_{\sigma'\to\sigma'}
\big\lVert \log(1+\s\,\Sop[f]/L)-
\log(1+\s\,\Sop[g]/L)\big\rVert_{\sigma'}.
\end{multline}
Similarly, the logarithm meets the Lipschitz bound 
\[
\big|\log(1+t)-\log(1+t)\big|\le 2|t-t'|
\]
whenever $t,t'$ are complex with $|t|,|t'|\le \frac12$. 
As a consequence, we find that
\[
\lVert \Rop(f,g)\rVert_{\sigma'}\le 6 |\s|\;
\big\lVert \Pop_\Omega\big\rVert_{\sigma'\to\sigma'}
\big\lVert \Sop\big\rVert_{\sigma\to \sigma'}
\big\lVert 1/L\big\rVert_{\sigma'}\big\lVert f-g\big\rVert_{\sigma}.
\]
This holds for $f,g$ in the ball of radius $r_0$ in
$\mathscr{H}_\sigma^\infty(\Gamma)$ and for
$\s\in\D(0, \frac1{\n_1})$,
where $\n_1$ is given above in \eqref{eq:n1}.
The proof is complete.
\end{proof}
\begin{rem}\label{rem:circular}
We remark that the condition \eqref{eq:n1} looks like it could be 
circular, since the right-hand side
depends on $\s$. However, if we initially fix 
$|\s|\le\epsilon_0$, then the quantity $\n_1$ will 
be uniformly
bounded for fixed $\sigma$ and $\sigma'$. Hence the
condition \eqref{eq:n1} will be met provided that $|\s|$ is small enough.
\end{rem}

\subsection{The local Lipschitz estimate for \texorpdfstring{$\Tope$}{}}
We proceed to the central statement of this section. We let $\n_1$ 
be the lower bound from Proposition~\ref{prop:bd-R} and recall the number $r_0$
from \eqref{eq:r0}.

\begin{thm}\label{thm:Lipschitz}
Let $\sigma$ and $\sigma'$ be positive real numbers with
\[
0<\sigma'<\sigma<1.
\]
For any $\s\in\D\big(0,\n_1^{-1}\big)$ and for any two functions $f$ and $g$ in the
ball of radius $r_0$ in $\mathscr{H}^\infty_\sigma(\Gamma)$, we have that
\begin{equation}
\lVert \Tope(f)-\Tope(g)\rVert_{\sigma'}\le
C(\sigma,\sigma')\,M_0(\sigma,\sigma'',\s)\,|\s|\,\lVert f-g\rVert_{\sigma},
\end{equation}
where the constant $C(\sigma,\sigma')$ is given by
\begin{multline}
C(\sigma,\sigma')
=\frac{3\big\lVert |\partial V|^{-2}\big\rVert_{\sigma'}\,
\big\lVert(1-|\phi|^2)/V\big\rVert_{\sigma'}}{2(\sigma-\sigma')}
\\
\Big(1+6\lVert \Pop_\Omega\rVert_{\sigma''\to\sigma''}
\sup_{|\s|\le \epsilon_0}\lVert 1/L\rVert_{\sigma''}
\exp\big(\lVert\Pop_\Omega[\log L]\rVert_{\sigma}\big) \Big),
\end{multline}
where $\sigma''$ is the mean $\sigma''=\frac12(\sigma+\sigma')$, 
and where $M_0(\sigma,\sigma'',\s)$
is as in \eqref{eq:bd-Sop-simple}.
\end{thm}

\begin{proof} 
Since $\mathscr{H}_{\sigma'}^\infty(\Gamma)$ is a Banach algebra, we have
\[
\lVert \Tope(f)-\Tope(g)\rVert_{\sigma'}\le \frac{1}{4}
\bigg\lVert\frac{1}{|\partial V|^2}\bigg\rVert_{\sigma'}\,
\bigg\lVert \frac{1-|\phi|^2}{V}\bigg\rVert_{\sigma'}
\bigg\lVert \frac{\s\,\Sop(f-g) + \e^{\Pop_\Omega[\log L]}\Rop(f,g)}{1-|\phi|^2}\bigg\rVert_{\sigma'},
\]
where have used the representation \eqref{eq:Top-diff} of $\Tope(f)-\Tope(g)$,
and where we recall that $\Rop$ was defined in \eqref{eq:discrepancy-f-g}.
If we introduce an intermediate parameter $\sigma''$ by 
\[
\sigma''=\frac12(\sigma+\sigma'),
\]
it follows from Lemma~\ref{lem:div} that
\begin{equation}
\lVert \Tope(f)-\Tope(g)\rVert_{\sigma'}
\le \frac{3}{2(\sigma''-\sigma')}
\bigg\lVert\frac{1}{|\partial V|^2}\bigg\rVert_{\sigma'}\,
\bigg\lVert \frac{1-|\phi|^2}{V}\bigg\rVert_{\sigma'}
\big\lVert \s\,\Sop(f-g)+\e^{\Pop_\Omega[\log L]}\Rop(f,g)\big\rVert_{\sigma''}.
\end{equation}
We conclude the proof by observing that in view of the triangle inequality, we have
\[
\big\lVert \s\,\Sop(f-g)+\e^{\Pop_\Omega[\log L]}\Rop(f,g)\big\rVert_{\sigma''}\le 
|\s|\;\lVert\Sop\rVert_{\sigma\to \sigma''}\lVert f-g\rVert_{\sigma}
+\e^{\lVert \Pop_\Omega[\log L]\rVert_{\sigma}}\lVert \Rop(f,g)\rVert_{\sigma''},
\]
so if we apply Proposition~\ref{prop:bd-R} to the 
last term on the right-hand side,
it follows that
\[
\big\lVert \Tope(f)-\Tope(g)\big\rVert_{\sigma'}\le C(\sigma,\sigma')\,|\s|\;
\lVert \Sop\rVert_{\sigma\to \sigma''}\lVert f-g\rVert_{\sigma}
\]
where 
\begin{multline}
C(\sigma,\sigma')=
\frac{3\big\lVert |\partial V|^{-2}\big\rVert_{\sigma'}\,
\big\lVert (1-|\phi|^2)/V\big\rVert_{\sigma'}
}{2(\sigma''-\sigma')}\\
\times \Big(1+ 6\lVert \Pop_\Omega\rVert_{\sigma''\to\sigma''}
\sup_{|\s|\le\epsilon_0}\lVert 1/L\rVert_{\sigma''}\exp\big(\lVert \Pop_\Omega[\log L]\rVert_{\sigma}\big)\Big)
\end{multline}
provided that $f$ and $g$ have norm at most $r_0$ in $\mathscr{H}_\sigma^\infty(\Gamma)$ and that
$|\s|\le \frac{1}{\n_1}$.
This proves the claim.
\end{proof}

Combining Theorem~\ref{thm:Lipschitz} with Proposition~\ref{prop:S-P-bd} above, 
we obtain the following immediate corollary which illustrates better what is going on.
For the formulation, we recall the positive number $\n_1$ defined above in \eqref{eq:n1}.

\begin{cor}
\label{cor:bd-T-Lip}
There exists a real number $\sigma^*$ with 
$0<\sigma^*<1$ such that whenever $0<\sigma'<\sigma<\sigma^*$, 
the Lipschitz bound 
\[
\big\lVert \Tope f - \Tope g\big\rVert_{\sigma'}\le 
M_1(\sigma,\sigma',\s)\,|\s|\;\lVert f-g\rVert_{\sigma}
\]
holds for $\s$ with $|\s|\le\frac{1}{\n_1}$ and $f,g$ in 
the ball of radius $r_0$ in $\mathscr{H}^\infty_{\sigma}(\Gamma)$,
where the constant $M_1(\sigma,\sigma',\s)$ takes the form
\[
M_1(\sigma,\sigma',\s)=C_1\Big(\frac{1}{(\sigma')^2(\sigma-\sigma')^2}
+\frac{|\s|}{(\sigma')^2(\sigma-\sigma')^4}\Big),
\]
for some positive constant $C_1$ depending only on $V$ and $\Omega$.
\end{cor}

\begin{proof}
If we plug in the definition of the constant $M_0(\sigma,\sigma'',\s)$
from \eqref{eq:bd-Sop-simple} into the estimate in Theorem~\ref{thm:Lipschitz}, 
we find that it only remains to verify that the quantities
\[
\big\lVert L^{-1}\big\rVert_{\sigma}, \qquad \big\lVert (\partial V)^{-2}\big\rVert_{\sigma}
\quad \text{and}\quad \big\lVert V^{-1}(1-|\phi|^2)\big\rVert_{\sigma},
\]
do not degenerate as $\sigma$ varies in the allowed parameter range.
But this follows from Proposition~\ref{prop:L-bd}.
\end{proof}

\subsection{Approximate solutions to the fixed point equation}
We recall the operator $\Tope$, defined above in \eqref{eq:Top-def}.
In this section, we show that iteration of the non-linear
operator $\Tope$ gives us an approximate fixed point,
after a suitable $\s$-dependent number of steps.

For the formulation, we let
\[
C_2=2\max\Big\{1,C_1\Big(\frac{16}{(\sigma^*)^4}+\frac{64}{(\sigma^*)^6}\Big)\Big\},
\]
where $C_1$ is the constant from Corollary~\ref{cor:bd-T-Lip}.
\begin{thm}
\label{thm:main-iteration}
There exists a positive constant $\varrho_0=\varrho(\hatR,\sigma^*)$, 
such that for all $\s$ in the disk $\D(0,\varrho_0)$
and for any positive integer $k$ with $k\le |\s|^{-1/2}C_2^{-1/2}$, we have
\[
\big\lVert \E_k-\Tope[\E_k]\big\rVert_{\frac12\sigma^*}
\le \big(\tfrac12C_2|\s|\,k^2\big)^k\lVert \E_{0}\rVert_{\sigma^*}.
\]
\end{thm}

In particular, by taking the diagonal restriction in the relevant 
bi-annulus and choosing $\mathcal{N}=\phi^{-1}(\A(\frac13\sigma^*))$
and letting, e.g., $k=\frac12|\s|^{-\frac12}C_2^{-\frac12}$, we obtain 
Lemma~\ref{lem:approx-fix-pt-prel} as a corollary. The factors $\frac13$ 
and $\frac12$ are not optimal, but are rather taken to have extra room for later.

\begin{proof}
Recall the definition \eqref{eq:r0} of the radius $r_0$:
\[
r_0=2\sup_{\s\in\D(0,\epsilon_0)}\lVert \E_0\rVert_{\sigma^*}.
\]
The Lipschitz estimates for $\Tope$ from \S\ref{s:est-Top} are then valid for all $f,g$
in the ball of radius $r_0$ in $\mathscr{H}^\infty_{\sigma}(\Gamma)$, provided that $\sigma\le\sigma^*$.
For a parameter $k\in\N$ which may depend on $|\s|$, we 
let $\sigma_{j}=\sigma_{j,k}$ be the numbers
\[
\sigma_j=\sigma^*-\frac{\sigma^* j}{2k},\qquad j=1,\ldots, k.
\]
Then $\sigma_j$ decreases with $j$, we have $\sigma_j-\sigma_{j+1}=\sigma^*/(2k)$,
while $\sigma_{k}=\sigma_{k,k}=\frac{1}{2}\sigma^*$ remains bounded away from 0.

In order for the estimates
from the previous section to apply, we need to require that 
\begin{equation}\label{eq:n-large-enough}
|\s|^{-1}\ge \n_1=
M_0(\sigma_j, \sigma_{j+1},\s)
\;\lVert 1/L\rVert_{\sigma_{j+1}}\max\Big\{2,
\lVert \Pop_\Omega \rVert_{\sigma_{j+1}\to\sigma_{j+1}}\Big\}r_0
\end{equation}
where we point out that the right-hand is expected to depend on $\s$
as a consequence of the choice of the parameters $\sigma_j$ and the $\s$-dependence
of $M_0(\sigma,\sigma',\s)$.
But in view of the exact formula for $M_0(\sigma,\sigma',\s)$ and the bounds given in
Proposition~\ref{prop:L-bd} and \eqref{eq:sigma-bound-P},
there exists some constant $C_3$
depending only on $\sigma^*$, $V$ and $\phi$, such that the quantity $\n_1$ given 
in \eqref{eq:n-large-enough} satisfies
\[
\n_1\le C_3\left(k +k^3|\s|\right).
\]
Hence, if we require that
$k\le |\s|^{-\frac12}C_2^{-\frac12}$, we obtain
\[
\n_1\le C_3(k + k^3|\s|)\le C_3 k(1+C_2^{-1}) \le \frac{C_3}{\sqrt{C_2}}(1+C_2^{-1})|\s|^{-\frac12}.
\]
It follows that if $\frac{C_3}{\sqrt{C_2}}(1+C_2^{-1})\le |\s|^{-\frac12}$, i.e., if
\[
|\s|\le \frac{C_2}{C_3^2}\frac{1}{(1+C_2^{-1})^2}=:\varrho_0,
\]
the inequality \eqref{eq:n-large-enough} holds.
In other words, the upper bound on $|\s|$, which is required for the 
estimates from the previous section to apply, is now
replaced by a $\s$-independent upper bound.

We let $k$ be a parameter with $k\le |\s|^{-\frac12}C_2^{-\frac12}$ and let 
$|\s|\le \varrho_0$.
Then, for any parameter $j=1,\ldots k$, we apply Corollary~\ref{cor:bd-T-Lip} 
to obtain the estimate
\[
\big\lVert \E_j-\Tope[\E_j]\big\rVert_{\sigma_j}
=\big\lVert \Tope[\E_{j-1}]-\Tope[\E_j]\big\rVert_{\sigma_j}
\le M_1(\sigma_j,\sigma_{j+1},\s)|\s|
\big\lVert \E_{j-1}-\E_j\big\rVert_{\sigma_{j+1}},
\]
where $M_1(\sigma,\sigma',\s)$ was defined in Corollary~\ref{cor:bd-T-Lip}.
Since $C_2\ge 2$ we have in particular the crude 
bound $k\le |\s|^{-\frac12}$, and hence we may 
combine the identity $\sigma_{j+1}-\sigma_j=\sigma^*/(2k)$ with the bound $\sigma_j\ge \frac12\sigma^*$
to get that
\begin{align}
M_1(\sigma_j,\sigma_{j+1},\s)&
=C_1\Big(\frac{1}{(\sigma_{j+1})^2(\sigma_j-\sigma_{j+1})^2}+
\frac{|\s|}{(\sigma_{j+1})^2(\sigma_j-\sigma_{j+1})^4}\Big)\\
&\le C_1\Big(\frac{16k^2}{(\sigma^*)^4}+
\frac{64|\s|k^4}{(\sigma^*)^6}\Big)\le \frac12 C_2k^2
\end{align}
Hence, we find that
\begin{equation}
\label{eq:iterative-step}
\big\lVert \E_j-\Tope[\E_j]\big\rVert_{\sigma_j}
\le \tfrac12 C_2|\s|\,k^2
\big\lVert \E_{j-1}-\Tope[\E_{j-1}]\big\rVert_{\sigma_{j-1}},
\end{equation}
provided that the norm on the right-hand side is at most $r_0$.

It remains to verify that the relevant norms of $\E_j$
remain bounded by $r_0$ throughout the iteration, and finally to apply the 
estimate \eqref{eq:iterative-step}
iteratively to estimate $\lVert \E_k\rVert_{\sigma_k}$
in terms of $\lVert \E_0\rVert_{\sigma_0}$.

\subsubsection*{The norms of $\E_j$}
We recall that $\E_0=\Tope[0]$ and that the radius $r_0$
was defined as $r_0=2\sup_{|\s|\le\epsilon_0}\lVert \E_0\rVert_{\sigma^*}$. 
If we write $\E_{-1}\equiv 0$,
we find that the estimate \eqref{eq:iterative-step} holds for $j=0$.
This forms the base step for an induction.

Suppose next that $\E_\ell$ lies in the ball of 
radius $r_0$ in $\mathscr{H}^\infty_{\sigma_{l}}$
for all $0\le \ell\le j-1$.
We may then estimate the norm of $\E_j$ from above by
\[
\lVert \E_j\rVert_{\sigma_{j}}\le \sum_{\ell=0}^j\lVert \E_\ell-
\E_{\ell-1}\rVert_{\sigma_\ell}=\sum_{\ell=0}^j\lVert \Tope \E_{\ell-1}-
\Tope \E_{\ell-2}\rVert_{\sigma_\ell}.
\]
If $j> |\s|^{-\frac12}C_2^{-\frac12}$ we are done, and else it follows that 
\begin{equation}
\lVert \E_j\rVert_{\sigma_{j}}\le \sum_{\ell=0}^{j}(\tfrac12 C_2 k^2 |\s|)^\ell
\lVert \E_0\rVert_{\sigma_0}\le \sum_{\ell=0}^{j}\frac1{2^\ell} 
\lVert \E_0\rVert_{\sigma_0}\le r_0.
\end{equation}
By induction, obtain the bound
\begin{equation}
\label{eq:Ek-bd}
\max_{0\le j\le k}\lVert \E_j\rVert_{\sigma_{j}}\le r_0
\end{equation}
for the full sequence of iterates.

\subsubsection*{The iteration of the estimate \eqref{eq:iterative-step}}\;
As a consequence of \eqref{eq:Ek-bd}, the estimate 
\eqref{eq:iterative-step} is in force for $j=0,1,2,\ldots,k$. Hence,
we may apply it iteratively, which yields
\begin{multline}
\label{eq:iteration-est}
\lVert \E_k-\Tope[\E_k]\rVert_{\frac12\sigma^*}=
\lVert \E_k-\Tope[\E_k]\rVert_{\sigma_k}\\
\le \tfrac12 C_2|\s| k^2\,\lVert \E_{k-1}-\Tope[\E_{k-1}\rVert_{\sigma_{k-1}}
\le\ldots\le \big(\tfrac12 C_2|\s|\,k^2\big)^k\lVert \E_0\rVert_{\sigma^*},
\end{multline}
and the proof is complete.
\end{proof}

\subsection{The auxiliary coefficient functions}
We let $\E_k=\Tope^{k+1}[0]$ be the
function from Theorem~\ref{thm:main-iteration}, and 
define $h_{k}$
in accordance with \eqref{eq:h-sol}, so that
\begin{equation}\label{eq:def-hkn}
h_{k}\eqd \frac12\Pop_\Omega\big[\log(L+\s \Sop [\E_{k}])\big]
-\frac{1}{4}\log\frac{\pi}{2}.
\end{equation}
We also let $H_{k}$ and $G_{k}$ be given by 
\[
H_{k}\eqd \log|\phi|^2+\s\Pop_{\Omega}[\E_{k}].
\]
and
\[
G_{k}\eqd \frac{\s(\Pop_\Omega-\Iop)[\E_{k}]+\log|\phi|^2}{2\hatR}
\]
respectively (cf.\ \eqref{eq:H-sol} and \eqref{eq:G-sol}),
where we point out that the division does not create any singularity, 
so that the functions $h_{k}$, $H_{k}$ and $G_{k}$ 
are all well-defined and real-analytically smooth on the 
neighborhood $\mathcal{N}$ of $\Gamma$, given by
\[
\mathcal{N}=\phi^{-1}\big(\mathbb{A}\big(\rho(\tfrac13\sigma^*)\big)\big),
\]
In fact, we can quantify this further.
\begin{prop}
\label{prop:bdd-coeff}
There exists a constant $C_4=C_4(\hatR,\phi,\sigma^*)$ such that for all $\s\in\D(0,\varrho_0)$ 
and positive integers $k$ with $k\le |\s|^{-\frac12}C_2^{-\frac12}$,
the functions $h_{k}$, $H_{k}$ and $G_{k}$
along with their gradients are bounded in modulus by $C_4$ on $\mathcal{N}$.
The same bound holds also for the harmonic conjugate $h^*_k$ of $h_k$ in $\Omega$.
\end{prop}

In particular, this implies the preliminary claim from Lemma~\ref{lem:bdd-coeff-prel}.

\begin{proof}
This proof relies on the observation that $\E_{k}$ lies in the ball 
of radius $r_0$ in $\mathscr{H}^\infty_{\frac12\sigma^*}(\Gamma)$ (see \eqref{eq:Ek-bd}), together with
Lemma~\ref{lem:div}, Proposition~\ref{prop:S-P-bd} and the equation \eqref{eq:d-bar-bd}.
We illustrate how to bound the complex gradient 
$\partial h_{k}$ on $\mathcal{N}$, and omit the 
verification of the other claims 
since they are very similar. We recall that
\[
h_{k}=\frac12\Pop_\Omega\big[\log(\s\,\Sop [\E_{k}]+L)\big]
-\frac{1}{4}\log\frac{\pi}{2}.
\]
Denote by $\sigma'$ the mid point between $\frac13\sigma^*$ and $\frac12\sigma^*$.
It follows from \eqref{eq:d-bar-bd} and \eqref{eq:sigma-bound-P} that 
\begin{multline}
\label{eq:est-d-bar-h}
\lVert \partial h_{k}\rVert_{\frac13\sigma^*}\le \frac{18}{3\sigma'-\sigma^*}\,
\lVert h_{k}\rVert_{\sigma'}
\lVert \phi'\rVert_{\frac13\sigma^*}+\frac{1}{4}\log\frac{\pi}{2}\\
\le \frac{54}{(3\sigma'-\sigma^*)\sigma'}
\big\lVert \log(\s\,\Sop [\E_{k}]+L)\big\rVert_{\sigma'}
\lVert \phi'\rVert_{\frac13\sigma^*}+\frac{1}{4}\log\frac{\pi}{2}
\end{multline}
But $\E_k$ is has norm at most $r_0$ in 
$\mathscr{H}^\infty_{\frac12\sigma^*}(\Gamma)$ 
and hence $\Sop[\E_k]$ 
is bounded in $\mathscr{H}^\infty_{\sigma'}(\Gamma)$, uniformly for $k$ and $\s$ subject
to the given requirements. Since moreover
$1/L\in \mathscr{H}^\infty_{\sigma^*}(\Gamma)$, we write
\[
\big\lVert \log(\s\,\Sop [\E_k]+L)\big\rVert_{\sigma'}\le
\big\lVert \log L\big\rVert_{\sigma'} 
+ \big\lVert \log(1+\s\, L^{-1}\Sop [\E_k])\big\rVert_{\sigma'}.
\]
Since $\varrho_0^{-1}\ge \n_1\ge \n_0$, we have $\s\in\D(0,\n_0^{-1})$, and hence 
$\lVert \s L^{-1}\Sop[\E_k]\rVert_{\sigma'}\le \frac12$
(see \eqref{eq:n0}).
As a consequence, we may use the elementary fact that $\log(1+t)\le 2t$ for $|t|\le 2$ together with 
Proposition~\ref{prop:S-P-bd} to obtain that
\begin{multline}
\label{eq:est-log-Sop+L}
\big\lVert \log(\s \Sop [\E_k]+L)\big\rVert_{\sigma'}\le
2\big\lVert \log L\big\rVert_{\sigma'} +2|\s|\;\lVert L^{-1}\rVert_{\sigma'}
\big\lVert \Sop [\E_k]\big\rVert_{\sigma'}+2\lVert L\rVert_{\sigma'}
\\ \le C\lVert \E_k\rVert_{\frac12\sigma^*} + 2\lVert L\rVert_{\frac12\sigma^*},
\end{multline}
where we suppress the exact formula for the constant $C$ (it may be found by inspecting
the constant in Proposition~\ref{prop:S-P-bd}). However, it is clear that we may take $C$
to depend only on $\hatR$ and $\sigma^*$.
The claim follows by putting together \eqref{eq:est-d-bar-h} and \eqref{eq:est-log-Sop+L}.
\end{proof}

\section{Truncation of asymptotic expansions}
\label{s:truncation}
In Theorem~\ref{thm:main}, we obtained an asymptotic formula for the planar orthogonal polynomials
of the form
\[
P=\n^{\frac14}(\phi')^{\frac12}\phi^n e^{h^\approx+\imag (h^\approx)^* + m\mathcal{Q}}
\Big(1+\Ordo(\e^{-c_0\sqrt{\n}})\Big), 
\]
where the error terms can either be interpreted in a suitable pointwise sense or in the space 
$L^2(\C,\e^{-2\n Q}\diffA)$.
The function $h^\approx=h_k$ is given by \eqref{eq:h-sol} in terms of the approximate solution 
$\E^\approx=\Tope^{k+1}[0]$
to the fixed-point equation $\Tope \E=\E$, where $k=k_\n$ is a large parameter of order $\Ordo(\n^{\frac12})$.
The function $h^\approx$ formally has some asymptotic expansion
\[
h^\approx\sim \hat{h}_{0} + \n^{-1} \hat{h}_{1} + \n^{-2} \hat{h}_2+\ldots,
\]
where the coefficients $\hat{h}_{j}$ for $j$ smaller than the iteration parameter $k$
are stable, meaning they do not change if we increase $\n$ and iterate further.
The goal of this section is to show that we can truncate the asymptotic expansion
at some appropriate $\n$-dependent finite term,
\[
\mathfrak{h}=\hat h_{0} + \n^{-1} \hat h_{1} + \n^{-2} \hat h_2+\ldots + \n^{-k}\hat h_{k}
\] 
and still have
\[
P=\n^{\frac14}(\phi')^{\frac12}\phi^n e^{\mathfrak{h}+\imag \mathfrak{h}^* + m\mathcal{Q}}
\Big(1+\e^{-c_0\sqrt{\n}}\Big),
\]
valid in the same region $\{z:\mathrm{d}_\C(z,\calS_\tau^c)\le \delta\n^{-\frac14}\}$. 

We begin with a lemma concerning the error term in the complex version of Taylor's formula.

\begin{lem}
\label{lem:Taylors-formula}
Let $f$ be a bounded analytic function on the unit disk $\D$. 
Denote by $\mathcal{P}_k f(z)$ the $k$-th
Taylor polynomial of $f(z)$. Then
\[
\left|\mathcal{P}_k f(z)-f(z)\right|\le |z|^{k+1} 
\Big(1+\frac{1}{\pi}\log\frac{1}{1-|z|^2}\Big)\;\lVert f\rVert_{H^\infty(\D)},\qquad z\in\D.
\]
\end{lem}

\begin{proof}
We begin with the Cauchy integral formula
\[
f(z)=\int_{\T}\frac{f(\eta)}{1-z\bar{\eta}}\diffs(\eta),\qquad z\in\D.
\]
The Taylor coefficients are similarly given by
\[
\widehat{f}(j)=\int_{\T}\bar\eta^{j}f(\eta)\diffs(\eta),
\]
so that
\[
\mathcal{P}_kf(z)=\sum_{j=0}^k z^j\int_{\T}\bar\eta^{j}f(\eta)\diffs(\eta).
\]
As a consequence, we get
\[
f(z)-\mathcal{P}_kf(z)=z^{k+1}\int_{\T}\frac{\bar\eta^{k+1}}{1-z\bar{\eta}}f(\eta)\diffs(\eta),
\]
which leads to the estimate
\[
\big|f(z)-\mathcal{P}_kf(z)\big|\le |z|^{k+1}\lVert f\rVert_{H^\infty(\D)}
\int_{\T}\frac{1}{|1-z\bar{\eta}|}\diffs(\eta).
\]
We rewrite the right-hand side integral as
\[
\int_{\T}\frac{1}{|1-z\bar{\eta}|}\diffs(\eta)=
\int_{\T}\big|(1-z\bar{\eta})^{-\frac12}\big|^2\diffs(\eta)=
\sum_{j=0}^\infty\frac{((\frac12)_j)^2}{(j!)^2}|z|^{2j}
=\sum_{j=0}^\infty \frac{\Gamma(\frac12+j)^2}{\Gamma(\frac12)^2\Gamma(j+1)^2}|z|^{2j}.
\]
In view of the estimate
\[
\Gamma\big(\tfrac12+j\big)\le \sqrt{j}\,\Gamma(j),\qquad j= 1,2,3,\ldots
\]
and the identity $\Gamma(\frac12)=\sqrt{\pi}$, we obtain
\[
\int_{\T}\frac{1}{|1-z\bar{\eta}|}\diffs(\eta)\le 1+\frac{1}{\pi}\sum_{j=1}^\infty\frac{|z|^{2j}}{j}
=1+\frac{1}{\pi}\log\frac{1}{1-|z|^2}.
\]
This completes the proof.
\end{proof}

Recall the operator $\Tope=\Tope_\s$ defined in \eqref{eq:Top-def},
and let $\E_k=\E_{k,\s}=\Tope^{k+1}[0]$. 
We recall also from the proof of Theorem~\ref{thm:main-iteration} that there exist
positive constants $\varrho_0=\varrho_0(V,\sigma^*)$
and $r_0=r_0(V,\sigma^*)$ such that $\E_{k}$ 
is holomorphic in $\s\in \D(0,\varrho_0)$ and such that
\begin{equation}
\label{eq:Ek-bd-app}
\lVert \E_j\rVert_{\frac12\sigma^*}\le r_0
\end{equation}
holds for $j=0,1,2,\ldots,k_\n$.

The following result implies Lemma~\ref{lem:truncation-h-prel}, if we specify to
$k'=k=k_\s=\frac12C_2^{-\frac12}|\s|^{-\frac12}$ and restrict to the ring domain 
$\mathcal{N}$.

\begin{lem}\label{lem:truncation}
There exists a positive constant $C_5=C_5(V,\sigma^*)$ such that 
for any $\s\in\D(0, \varrho_0)$ and any positive integer 
$k$ with $k\le |\s|^{-\frac12}C_2^{-\frac12}$ and any positive integer $k'$,
we have that
\[
\big\lVert\mathcal{P}_{k'}h_{k}-h_{k}
\big\rVert_{\frac13\sigma^*}\le C_5 \big(C_2|\s|k^2\big)^{k'+1}.
\]
\end{lem}

\begin{proof}
By repeating the proof of Proposition~\ref{prop:bdd-coeff}, in particular
by \eqref{eq:est-log-Sop+L}, we have
\[
\lVert h_{k}\rVert_{\frac13\sigma^*}\le 
C|\s|\,\lVert \E_k\rVert_{\frac12\sigma^*}+
2\lVert L\rVert_{\frac{1}{2}\sigma^*},
\]
where $C=C(\hatR,\phi,\sigma^*)$ is a positive constant.
By the bound \eqref{eq:Ek-bd-app} we have that $\E_k$ has norm
at most $r_0$ in $\mathscr{H}^\infty_{\frac12\sigma^*}(\Gamma)$,
whence
\[
\lVert h_{k}\rVert_{\frac13\sigma^*}\le C|\s| r_0
+2\sup_{|\s|\in\D(0,\varrho_0)}\lVert L\rVert_{\frac12\sigma^*}\le C_5
\]
for some constant $C_5=C_5(\hatR,\sigma^*)$ when $k$ and $\s$ satisfy
the given requirements.
In other words, as a function of $(z,\s)$, $h_{k}$
belongs to the tensor product space 
\[
\mathscr{H}^\infty_{\frac13 \sigma^*}(\Gamma)\otimes H^\infty(\D(0,C_2^{-1} k^{-2})),
\]
and has norm at most $C'$. Here, the tensor product space should be understood
as the class of functions $f(z,\s)$ which are holomorphic in $\s\in\D(0,\epsilon_1 k^{-2})$,
such that for fixed $\s$, we have $f(\cdot,\s)\in \mathscr{H}^\infty_\sigma(\Gamma)$.
But then, by Lemma~\ref{lem:Taylors-formula} and a simple scaling argument,
the Taylor polynomial $\mathcal{P}_{k'} h_{k}$ of $h_{k}$
in the variable $\s$ satisfies
\[
\big\lVert\mathcal{P}_{k'} h_{k}-h_{k}\big\rVert_{\frac13\sigma^*}\le 
C_5\big(C_2|\s|k^2\big)^{k'+1},
\]
for $\s$ in the smaller disk $\D(0,\frac12C_2^{-1} k^{-2})$. 
This completes the proof.
\end{proof}

\subsection*{Acknowledgement of funding}
H.\ H.\ acknowledges support from Vetenskapsr{\aa}det (VR grant 2020-03733), the Leverhulme
trust (grant VP1-2020-007), and by grant 075-15-2021-602 of the Government of the
Russian Federation for the state support of scientific research, carried out under the
supervision of leading scientists. 
A.\ W.\ acknowledges support from the Knut and Alice Wallenberg Foundation (grant 017.0389)

\end{document}